\documentclass[12pt,reqno]{amsart} 
\usepackage{amssymb} 
\usepackage{enumerate}
\usepackage{mathrsfs}
\usepackage[all]{xy}
\usepackage{bm}
\usepackage{dsfont}
\usepackage[colorlinks=true, citecolor=red, linkcolor=blue]{hyperref}
\usepackage{rotating}
\usepackage{bold-extra}

\usepackage[usenames]{color}

\makeatletter
\@namedef{subjclassname@2020}{\textup{2020} Mathematics Subject 
Classification} 
\makeatother

\let\Gamma=\varGamma
\let\Omega=\varOmega
\let\Sigma=\varSigma

\newlength{\short}
\newlength{\shorter}
\newcommand{\trs}[1]{\mathfrak{tr}_{#1}}
\newcommand{\trsh}[2][]{\trs{#1h_{#2}}}

\newcommand{\Grf}{\mathfrak{Gr}}
\newcommand{\grf}[1]{\mathfrak{gr}_{#1}}
\newcommand{\grfh}[2][]{\grf{#1h_{#2}}}

\newcommand{\ttt}{\mbf{\tau}}

\newcommand{\Mat}[2]{\text{\boldd{$M_{#1#2}$}}}

\newcommand{\dTodd}[2]{\text{\boldd{$\textit{Gol}_{#1#2}$}}}

\newcommand{\LL}[1]{\def\test{#1}\def\tst{2}\textit{\textbf{L}}%
\ifx\test\tst\else^{(1#1)}\fi}
\newcommand{\NN}[1]{\def\test{#1}\def\tst{2}\textit{\textbf{N}}%
\ifx\test\tst\else^{(1#1)}\fi}
\def\Qtrp[#1,#2,#3]{{\ll}{#1},#2,#3{\gg}}  
\def\trp[#1,#2,#3]{{[\![}#1,#2,#3{]\!]}}

\definecolor{darkgreen}{rgb}{0,0.5,0}
\definecolor{bluegreen}{rgb}{0,0.2,0.8}
\definecolor{darkred}{rgb}{0.8,0,0}
\definecolor{newercolor}{rgb}{0.2,0,1}
\definecolor{darkyellow}{rgb}{0.7,0.7,0}
\definecolor{orange}{rgb}{0.9,0.4,0}

\newcommand{\bmid}{\mathrel{\big|}}

\let\Gamma=\varGamma
\let\Omega=\varOmega

\newcommand{\EE}[2][]{\mathbf{E}_{#2}}

\newcommand{\0}[1]{\textup{\textbf{\uppercase{#1}}}}

\newcommand{\3}[1]{\textup{\textbf{3\uppercase{#1}}}}
\newcommand{\4}[1]{\overline{#1}}   %%{\widebar{#1}}
\newcommand{\5}[1]{\widehat{#1}}

\newcommand{\too}{\longrightarrow}

\newcommand{\boldd}[1]{{\mathversion{bold}\textbf{#1}}}

\newcommand{\mbf}[1]{\text{\boldmath $#1$\unboldmath}}

\newcommand{\lie}[3]{\def\test{#2}\def\tst{G}\ifx\test\tst{{}^{#1}#2_{#3}}
\else{{}^{#1}\!#2_{#3}}\fi}

\let\oldcirc=\circ
\renewcommand{\circ}{\mathchoice
    {\mathbin{\scriptstyle\oldcirc}}{\mathbin{\scriptstyle\oldcirc}}
    {\mathbin{\scriptscriptstyle\oldcirc}}
    {\mathbin{\scriptscriptstyle\oldcirc}}}

\newlength{\upto}\newlength{\dnto}
\SelectTips{cm}{10} \UseTips   %% use computer modern tips in xy

\numberwithin{equation}{section}

\def\beq#1\eeq{\begin{equation*}#1\end{equation*}}
\def\beqq#1\eeqq{\begin{equation}#1\end{equation}}

\let\emptyset=\varnothing
\renewcommand{\:}{\colon}   %% as in f:X-->Y
\newcommand{\longline}{\bigskip\centerline{\hbox to 5cm{\hrulefill}}\bigskip}

\newcommand{\mxtwo}[4]{\left(\begin{smallmatrix}#1&#2\\#3&#4
\end{smallmatrix}\right)}
\newcommand{\mxthree}[9]{\left(\begin{smallmatrix}#1&#2&#3\\#4&#5&#6\\
#7&#8&#9\end{smallmatrix}\right)}
\newcommand{\Mxtwo}[4]{\begin{pmatrix}#1&#2\\#3&#4\end{pmatrix}}

\newcommand{\coltwo}[2]{\left(\begin{smallmatrix}#1\\#2
\end{smallmatrix}\right)}

\newcommand{\colthree}[3]{\left(\begin{smallmatrix}#1\\#2\\#3
\end{smallmatrix}\right)}
\newcommand{\Coltwo}[2]{\begin{pmatrix}#1\\#2\end{pmatrix}}

\newcommand{\Colthree}[3]{\begin{pmatrix}#1\\#2\\#3\end{pmatrix}}

\newcommand{\mxfoura}[8]{\left(\begin{smallmatrix}#1&#2&#3&#4\\#5&#6&#7&#8\\}
\newcommand{\mxfourb}[8]{#1&#2&#3&#4\\#5&#6&#7&#8\end{smallmatrix}\right)}

\DeclareMathAlphabet\EuR{U}{eur}{m}{n}
\SetMathAlphabet\EuR{bold}{U}{eur}{b}{n}

\newcommand{\higherlim}[2]{\displaystyle\setbox1=\hbox{\rm lim}
	\setbox2=\hbox to \wd1{\leftarrowfill} \ht2=0pt \dp2=-1pt
	\setbox3=\hbox{$\scriptstyle{#1}$}
	\def\test{#1}\ifx\test\empty
	\mathop{\mathop{\vtop{\baselineskip=5pt\box1\box2}}}\nolimits^{#2}
	\else
	\ifdim\wd1<\wd3
	\mathop{\hphantom{^{#2}}\vtop{\baselineskip=5pt\box1\box2}^{#2}}_{#1}
	\else
	\mathop{\mathop{\vtop{\baselineskip=5pt\box1\box2}}_{#1}}%
	\nolimits^{#2}
	\fi\fi}
\newcommand{\higherlimm}[2]{\setbox1=\hbox{\rm lim}
	\setbox2=\hbox to \wd1{\leftarrowfill} \ht2=0pt \dp2=-1pt
	\mathop{\mathop{\vtop{\baselineskip=5pt\box1\box2}}}\limits_{#1}
	\nolimits^{#2}}

\newcounter{let} 
\loop\stepcounter{let}
\expandafter\edef\csname cal\alph{let}\endcsname%
{\noexpand\mathcal{\Alph{let}}}
\ifnum\thelet<26\repeat

\loop\stepcounter{let}
\expandafter\edef\csname scr\alph{let}\endcsname%
{\noexpand\mathscr{\Alph{let}}}
\ifnum\thelet<26\repeat

\loop\stepcounter{let}
\expandafter\edef\csname frak\alph{let}\endcsname%
{\noexpand\mathfrak{\Alph{let}}}
\ifnum\thelet<26\repeat

\newcommand{\tdef}[2][]{\expandafter\newcommand\csname#2\endcsname%
{#1\textup{#2}}}
\tdef{Iso}   \tdef{Aut}    \tdef{Out}    \tdef{Inn}    \tdef{Hom}   
\tdef{Ext}   \tdef{Irr}    \tdef{rad}    \tdef{Sym}    \tdef{ord}
\tdef{End}   \tdef{Inj}    \tdef{map}    \tdef{Ker}    \tdef{Ob}    
\tdef{Mor}   \tdef{Res}    \tdef{Spin}   \tdef{Id}     \tdef{wt}
\tdef{rk}    \tdef{conj}   \tdef{incl}   \tdef{proj}   \tdef{diag} 
\tdef{trf}   \tdef{Sol}     \tdef{cj}    \tdef{Outdiag} \tdef{lcm}
\tdef{free}  \tdef{supp}   \tdef{soc}    \tdef{Ind}    \tdef{Tr}
\tdef[_]{typ}   \tdef[^]{op}    \tdef[^]{ab}  \tdef{Ree}
\tdef{Mon}   \tdef{Perm}   \tdef{Herm}  
\tdef{Aff}

\newcommand{\fdef}[1]{\expandafter\newcommand\csname#1\endcsname%
{\mathfrak{#1}}}
\fdef{X}  \fdef{red}  \fdef{foc}  \fdef{hyp}

\newcommand{\bbdef}[1]{\expandafter\newcommand% 
\csname#1\endcsname{\mathbb{#1}}}
\bbdef{C} \bbdef{F} \bbdef{R} \bbdef{Z} \bbdef{Q}

\newcommand{\itdef}[1]{\expandafter\newcommand\csname#1\endcsname%
{\textit{#1}}}
\itdef{PSL}  \itdef{PSU}  \itdef{SL}  \itdef{SU}  \itdef{GL} \itdef{GU}
\itdef{UT}   \itdef{Sp}   \itdef{PSp} \itdef{PGL} \itdef{GO} \itdef{Co}
\itdef{SO}   \itdef{SD}   \itdef{Th}  \itdef{He}  \itdef{Sz} \itdef{AGL}
\itdef{Suz}  \itdef{McL}  \itdef{Ly}  \itdef{HS}  \itdef{Ru} \itdef{Fi}

\newcommand{\gee}{\varepsilon}

\newcommand{\gen}[1]{\langle{#1}\rangle}
\newcommand{\Gen}[1]{\bigl\langle{#1}\bigr\rangle}

\let\nsg=\normal
\let\nnsg=\ntrianglelefteq

\newcommand{\syl}[2]{\textup{Syl}_{#1}(#2)}
\newcommand{\sylp}[1]{\syl{p}{#1}}
\newcommand{\autf}[1][]{\Aut_{\calf_{#1}}}
\newcommand{\outf}[1][]{\Out_{\calf_{#1}}}
\newcommand{\homf}[1][]{\Hom_{\calf_{#1}}}
\newcommand{\isof}[1][]{\Iso_{\calf_{#1}}}

\newcommand{\sminus}{\smallsetminus}
\newcommand{\defeq}{\overset{\textup{def}}{=}}

\newcommand{\longleft}[1]{\;{\leftarrow%
\count255=0 \loop \mathrel{\mkern-6mu}%
    \relbar\advance\count255 by1\ifnum\count255<#1\repeat}\;}
\newcommand{\longright}[1]{\;{\count255=0 \loop \relbar\mathrel{\mkern-6mu}%
    \advance\count255 by1\ifnum\count255<#1\repeat\rightarrow}\;}
\newcommand{\Right}[2]{\overset{#2}{\longright#1}}
\newcommand{\RIGHT}[3]{\mathrel{\mathop{\kern0pt\longright#1}
	\limits^{#2}_{#3}}}

\newcommand{\LEFT}[3]{\mathrel{\mathop{\kern0pt\longleft#1}\limits^{#2}_{#3}}
}
\newcommand{\longleftright}[1]{\;{\leftarrow\mathrel{\mkern-6mu}%
    \count255=0\loop\relbar\mathrel{\mkern-6mu}% 
    \advance\count255 by1\ifnum\count255<#1\repeat\rightarrow}\;} 
 
\newcommand{\onto}[1]{\;{\count255=0 \loop \relbar\joinrel
    \advance\count255 by1
    \ifnum\count255<#1 \repeat \twoheadrightarrow}\;}

\newcommand{\RLEFT}[3]{\mathrel{%
   \mathop{\vcenter{\baselineskip=0pt\hbox{$\kern0pt\longright#1$}%
   \hbox{$\kern0pt\longleft#1$}}}\limits^{#2}_{#3}}}

\numberwithin{table}{section}

\newenvironment{Table}[1][]{\stepcounter{equation}\begin{table}[#1]}
{\end{table}}

\renewenvironment{enumerate}[1][]
{\begin{enumerat}[#1]\setlength{\itemsep}{6pt}}{\end{enumerat}}

\renewenvironment{itemize}[1][-15]
{\begin{itemiz}\setlength{\itemsep}{6pt}\setlength{\itemindent}{#1pt}}
{\end{itemiz}}

\newenvironment{enuma}{\begin{enumerate}[{\rm(a) }]}{\end{enumerate}}
\newtheorem{Thm}{Theorem}[section]
\newtheorem{Prop}[Thm]{Proposition}
\newtheorem{Cor}[Thm]{Corollary}
\newtheorem{Lem}[Thm]{Lemma}
\newtheorem{Claim}[Thm]{Claim}
\newtheorem{Ass}[Thm]{Assumption}
\newtheorem{Conj}[Thm]{Conjecture}
\newtheorem{Not}[Thm]{Notation}
\newtheorem{Hyp}[Thm]{Hypotheses}

\newtheorem{Thmm}{Theorem}

\theoremstyle{definition}
\newtheorem{Defi}[Thm]{Definition} 
\newtheorem{Rmk}[Thm]{Remark}
\newtheorem{Ex}[Thm]{Example}

\theoremstyle{remark}

\def\Qpr[#1,#2]{[\![#1,#2]\!]}

\newcommand{\UU}[2]{\scru_{#1}(#2)}
\newcommand{\VV}[2]{\scrt_{#1}(#2)}
\newcommand{\WW}[2]{\scrw_{#1}(#2)}
\newcommand{\RR}[3][]{\scrr^{#1}_{#2}(#3)}  %%{\scrr_{#1}(#2,#3)}
\newcommand{\hRR}[3][]{\5\scrr^{#1}_{#2}(#3)}  %%{\5\scrr_{#1}(#2,#3)}
\title{Nonrealizability of certain representations in fusion systems}

\author{Bob Oliver}
\address{Universit\'e Sorbonne Paris Nord, LAGA, UMR 7539 du CNRS, 
99, Av. J.-B. Cl\'ement, 93430 Villetaneuse, France.}
\email{bobol@math.univ-paris13.fr}
\thanks{B. Oliver is partially supported by UMR 7539 of the CNRS. Part of 
this work was carried out at the Isaac Newton Institute for Mathematical 
Sciences during the programme GRA2, supported by EPSRC grant nr. 
EP/K032208/1.}


\begin{document}

\begin{abstract} 
For a finite abelian $p$-group $A$ and a subgroup $\Gamma\le\Aut(A)$, we 
say that the pair $(\Gamma,A)$ is fusion realizable if there is a saturated 
fusion system $\calf$ over a finite $p$-group $S\ge A$ such that 
$C_S(A)=A$, $\autf(A)=\Gamma$ as subgroups of $\Aut(A)$, and $A\nnsg\calf$. 
In this paper, we develop tools to show that certain representations are 
not fusion realizable in this sense. For example, we show, for $p=2$ or $3$ 
and $\Gamma$ one of the Mathieu groups, that the only $\F_p\Gamma$-modules 
that are fusion realizable (up to extensions by trivial modules) are the 
Todd modules and in some cases their duals. 
\end{abstract}

\maketitle

Fix a prime $p$. A saturated fusion system over a finite $p$-group $S$ is a 
category whose objects are the subgroups of $S$, and whose morphisms are 
injective homomorphisms between those subgroups that satisfy certain axioms 
formulated by Puig \cite{Puig}, motivated in part by the Sylow theorems for 
finite groups. See Definition \ref{d:s.f.s.} for more details. 

Consider a pair $(\Gamma,A)$, where $A$ is a finite abelian $p$-group and 
$\Gamma\le\Aut(A)$ is a group of automorphisms. We say that $(\Gamma,A)$ is 
\emph{fusion realizable} if there is a saturated fusion system $\calf$ over 
some finite $p$-group $S\ge A$ such that $C_S(A)=A$, $A\nnsg\calf$, and 
$\autf(A)=\Gamma$ as groups of automorphisms of $A$. We also say that 
$(\Gamma,A)$ is \emph{realized by $\calf$} in this situation. 

In an earlier paper \cite{O-todd}, we considered the special case where 
$p=3$, $O^{3'}(\Gamma)\cong2M_{12}$, $M_{11}$, or $A_6$, and $A$ is an 
elementary abelian $3$-group of rank $6$, $5$, or $4$, respectively, and 
classified the saturated fusion systems that realize some pair $(\Gamma,A)$ 
of this form. In this paper, we take the opposite approach, and develop 
tools that we use to show that ``most'' $\F_p\Gamma$-modules are not fusion 
realizable; i.e., cannot be realized by any saturated fusion system. 

For example, in Definition \ref{d:K&R} and Proposition 
\ref{p:not.str.cl.3a}, we define certain sets $\RR{T}{A}$, for $A$ an 
abelian $p$-group and $T\le\Aut(A)$ a $p$-subgroup, with the property that 
$\RR{T}{A}\ne\emptyset$ if there is a fusion realizable pair $(\Gamma,A)$ 
where $T\in\sylp{\Gamma}$. As one of the consequences of this proposition, 
we show (Corollary \ref{c:not.str.cl.2}) that if $A$ is elementary abelian 
and $(\Gamma,A)$ is fusion realizable, then there is $m\ge1$ and an 
elementary abelian $p$-subgroup $B\le\Gamma$ of rank $m$ such that for each 
$g\in B^\#$, the action of $g$ on $A$ has at most $m$ nontrivial Jordan 
blocks.

Theorems \ref{ThA} and \ref{ThB} as stated below are our main applications 
so far of these tools. For example, as one special case of Theorem 
\ref{ThA}, we show that the Golay modules for $M_{22}$ and $M_{23}$ are not 
fusion realizable. In contrast, the Todd modules for $M_{22}$ and $M_{23}$ 
(dual to the Golay modules) are realized by the fusion systems of the 
Fischer groups $\Fi_{22}$ and $\Fi_{23}$, and the Golay module for 
$\Aut(M_{22})$ (a case not covered by the statement of Theorem \ref{ThA}) 
is realized by the fusion system of the Conway group $\Co_2$.

\begin{Thmm}[Theorem \ref{t:M11-24}] \label{ThA}
Fix a prime $p$, and let $\Gamma$ be a finite group such that 
$\Gamma_0=O^{p'}(\Gamma)$ is quasisimple and $\Gamma_0/Z(\Gamma_0)$ is one 
of Mathieu's five sporadic groups. Let $A$ be an 
$\F_p\Gamma$-module such that $(\Gamma,A)$ is fusion realizable, 
and set $A_0=[\Gamma_0,A]/C_{[\Gamma_0,A]}(\Gamma_0)$. Then either 
\begin{itemize} 

\item $p=2$, and $A_0$ is the Todd module for 
$\Gamma\cong M_{22}$, $M_{23}$, or $M_{24}$ or the Golay module for 
$\Gamma\cong M_{24}$; or 

\item $p=3$, $\Gamma\cong M_{11}$, $M_{11}\times C_2$, or $2M_{12}$, and 
and $A_0$ is the Todd module or Golay module for $\Gamma_0$; or 

\iffalse
$\Gamma\cong M_{11}$ or $2M_{12}$; or 
\fi

\item $p=11$, $\Gamma_0\cong 2M_{12}$ or $2M_{22}$, 
$\Gamma/Z(\Gamma_0)\cong\Aut(M_{12})\times C_5$ or $\Aut(M_{22})\times 
C_5$, and $A_0$ is a $10$-dimensional simple $\F_{11}\Gamma$-module.

\end{itemize}
\end{Thmm}

When $p=2$ or $3$, the nonrealizability of $(\Gamma,A)$ in Theorem 
\ref{ThA} is shown in all cases by proving that the set $\scrr_T(A)$ 
mentioned above is empty for $T\in\sylp\Gamma$. For $p>3$, it follows from 
results in \cite{indp2}.

Theorem \ref{ThB} is a restatement of a theorem of O'Nan \cite[Lemma 
1.10]{ONan} in the context of fusion systems, included here to illustrate 
how these methods apply when $A$ is not elementary abelian. Its proof is 
similar to O'Nan's, but is shortened by using results in Section 
\ref{s:A<|F}. 

\begin{Thmm}[Theorem \ref{t:Alp}] \label{ThB}
Assume, for some $n\ge3$, that $A=\gen{v_1,v_2,v_3}\cong C_{2^n}\times 
C_{2^n}\times C_{2^n}$, and that $S=A\gen{s,t}$ is an extension of $A$ by 
$D_8$ with action as described in Table \ref{tbl:D8onA}. Then $A$ is normal in 
every saturated fusion system over $S$. Thus there is no $\Gamma\le\Aut(A)$ 
with $\Aut_S(A)\in\syl2{\Gamma}$ such that $(\Gamma,A)$ is fusion 
realizable.
\end{Thmm}

The paper is organized as follows. 
After summarizing in Section \ref{s:background} the basic definitions and 
properties of fusion systems that will be needed, we state and prove our 
main criteria for fusion realizability in Section \ref{s:A<|F}. We then 
look at representations of Mathieu groups in Section \ref{s:Mathieu} and 
prove Theorem \ref{ThA} (Theorem \ref{t:M11-24}), and study Alperin's 
$2$-groups in Section \ref{s:Alp} and prove Theorem \ref{ThB} (Theorem 
\ref{t:Alp}). We finish with three appendices: Appendix \ref{s:JV(x)} with 
some general results on representations, and Appendices \ref{s:Todd-F2} and 
\ref{s:3M22} where we set up notation to work with the Golay modules for 
$M_{22}$ and $M_{23}$, and the $6$-dimensional $\F_43M_{22}$-module, 
respectively.

\bigskip

\noindent\textbf{Notation and terminology:} Most of our notation for 
working with groups is fairly standard. When $P\le G$ and $x\in N_G(P)$, we 
let $c_x^P\in\Aut(P)$ denote conjugation by $x$ on the left: 
$c_x^P(g)=\9xg=xgx^{-1}$. Also, $\sylp{G}$ is the set of Sylow 
$p$-subgroups of a finite group $G$, and $G^\#=G\sminus\{1\}$. Other 
notation used here includes: \begin{itemize} 

\item $E_{p^m}$ is always an elementary abelian $p$-group of rank $m$; 

\item $A\rtimes B$ and $A.B$ denote a semidirect product and an 
arbitrary extension of $A$ by $B$; and 

\item $2M_{12}$, $nM_{22}$, and $2A_4$ denote (nonsplit) central extensions 
of $C_2$ or $C_n$ by the groups $M_{12}$, $M_{22}$, or $A_4$, respectively. 

\end{itemize}
Also, composition of functions and homomorphisms is always written from 
right to left.

\bigskip

\noindent\textbf{Thanks:} The author would like to thank the Newton 
Institute in Cambridge for its hospitality while he was finishing the 
writeup of this paper. He would also like to thank the referee for 
carefully reading the paper and making several suggestions for 
improvements.

\section{Background definitions and results}
\label{s:background}

We recall here some of the basic definitions and properties of 
saturated fusion systems. Our main reference is \cite{AKO}, although most 
of the results are also shown in \cite{Craven}. 

A \emph{fusion system} $\calf$ over a finite $p$-group $S$ is a 
category whose objects are the subgroups of $S$, such that for each $P,Q\le 
S$, 
\begin{itemize}
\item $\Hom_S(P,Q) \subseteq\homf(P,Q)\subseteq\Inj(P,Q)$; and 

\item every morphism in $\calf$ is the composite of an $\calf$-isomorphism 
followed by an inclusion.

\end{itemize} 
Here, $\Hom_S(P,Q) = \{c_g\in\Hom(P,Q) \,|\, g\in S,~ \9gP\le Q \}$. We 
also write $\isof(P,Q)$ for the set of $\calf$-isomorphisms from $P$ to 
$Q$, and $\autf(P)=\isof(P,P)$.

In order for fusion systems to be very useful, we need to assume they satisfy 
the following saturation properties, motivated by the Sylow theorems and 
first formulated by Puig \cite{Puig}. 

\begin{Defi} \label{d:s.f.s.}
Let $\calf$ be a fusion system over a finite $p$-group $S$.
\begin{enuma}
\item Two subgroups $P,Q\le{}S$ are \emph{$\calf$-conjugate} if 
$\isof(P,Q)\ne\emptyset$, and two elements $x,y\in S$ are 
$\calf$-conjugate if there is $\varphi\in\homf(\gen{x},\gen{y})$ such that 
$\varphi(x)=y$. The $\calf$-conjugacy classes of $P\le S$ and $x\in S$ are 
denoted $P^\calf$ and $x^\calf$, respectively. 

\item A subgroup $P\le S$ is \emph{fully normalized} in $\calf$ 
(\emph{fully centralized} in $\calf$) if $|N_S(P)|\ge|N_S(Q)|$ for each 
$Q\in P^\calf$ ($|C_S(P)|\ge|C_S(Q)|$ for each $Q\in P^\calf$). 

\item The fusion system $\calf$ is \emph{saturated} if it satisfies the 
following two conditions:\medskip
\begin{itemize}

\item \textup{(Sylow axiom)} For each subgroup $P\le S$ fully normalized in 
$\calf$, $P$ is fully centralized and $\Aut_S(P)\in\sylp{\autf(P)}$.

\item \textup{(extension axiom)} For each isomorphism $\varphi\in\isof(P,Q)$ 
in $\calf$ such that $Q$ is fully centralized in $\calf$, $\varphi$ extends 
to a morphism $\4\varphi\in\homf(N_\varphi,S)$ where 
	\[ N_\varphi = \{ g\in N_S(P) \,|\, \varphi c_g \varphi^{-1} \in 
	\Aut_S(Q) \}. \]
\end{itemize}
\end{enuma}
\end{Defi}

Definition \ref{d:s.f.s.} is the definition first given in \cite{BLO2}, and 
is used here since it seems to be the easiest to apply for our purposes. It 
is slightly different from that given in \cite[Definition I.2.2]{AKO}, but 
the two are equivalent by \cite[Proposition I.2.5]{AKO}. Its equivalence 
with Puig's original definition is shown in \cite[Proposition I.9.3]{AKO}.

As one example, the fusion system of a finite group $G$ with respect to a Sylow 
$p$-subgroup $S\le G$ is the category $\calf_S(G)$ whose objects are the 
subgroups of $S$, and whose morphisms are those homomorphisms between 
subgroups that are induced by conjugation in $G$. It is clearly a fusion 
system and was shown by Puig to be saturated. (See \cite[Proposition 
1.3]{BLO2} for a proof of saturation in terms of Definition 
\ref{d:s.f.s.}.)

We will also need to work with certain classes of subgroups in a fusion 
system. Recall, for a pair of finite groups $H<G$, that $H$ is \emph{strongly 
$p$-embedded in $G$} if $p\bmid|H|$, and $p\nmid|H\cap\9gH|$ for $g\in 
G\sminus H$.

\begin{Defi} \label{d:subgroups}
Let $\calf$ be a fusion system over a finite $p$-group $S$. For 
$P\le S$, 
\begin{itemize}

\item $P$ is \emph{$\calf$-centric} if $C_S(Q)\le Q$ for each $Q\in P^\calf$; 

\item $P$ is \emph{$\calf$-essential} if $P$ is $\calf$-centric and fully 
normalized in $\calf$ and the group $\outf(P)=\autf(P)/\Inn(P)$ contains a 
strongly $p$-embedded subgroup; 

\item $P$ is \emph{weakly closed in $\calf$} if $P^\calf=\{P\}$; 

\item $P$ is \emph{strongly closed in $\calf$} if for each $x\in P$, 
$x^\calf\subseteq P$; 

\item $P$ is \emph{central} in $\calf$ if each 
$\varphi\in\homf(Q,R)$, for $Q,R\le S$, extends to some 
$\4\varphi\in\homf(QP,RP)$ such that $\4\varphi|_P=\Id_P$; and 

\item $P$ is \emph{normal in $\calf$} ($P\nsg\calf$) if each morphism in $\calf$ 
extends to a morphism that sends $P$ to itself.
\end{itemize}
We also let $\calf^c$ and $\EE\calf$ be the sets of subgroups of $S$ that 
are $\calf$-centric or $\calf$-essential, respectively.
\end{Defi}

The following is one version of the Alperin-Goldschmidt fusion theorem for 
fusion systems. 

\begin{Thm}[{\cite[Theorem I.3.6]{AKO}}] \label{t:AFT}
Let $\calf$ be a saturated fusion system over a finite $p$-group $S$. Then 
each morphism in $\calf$ is a composite of restrictions of automorphisms 
$\alpha\in\autf(R)$ for $R\in\EE\calf\cup\{S\}$. 
\end{Thm}

The next proposition is more technical. 

\begin{Prop}[{\cite[Lemma I.2.6(c)]{AKO}}] \label{p:Hom(NSP,S)}
Let $\calf$ be a saturated fusion system over a finite $p$-group $S$. Then 
for each $P\le S$, and each $Q\in P^\calf$ fully normalized in $\calf$, 
there is $\psi\in\homf(N_S(P),S)$ such that $\psi(P)=Q$.
\end{Prop}

Normal $p$-subgroups in a fusion system are strongly closed, but the 
converse does not always hold. The following is one situation where it does 
hold. For a much more detailed list of conditions under which strongly 
closed subgroups in a fusion system are normal, see \cite[Theorem 
B]{Kizmaz}.

\begin{Lem}[{\cite[Corollary I.4.7(a)]{AKO}}] \label{l:s.cl.=>normal}
Let $\calf$ be a saturated fusion system over a finite $p$-group $S$. If 
$A\nsg S$ is an abelian subgroup that is strongly closed in $\calf$, then 
$A\nsg\calf$. 
\end{Lem}

We next look at centralizers of $p$-subgroups in fusion systems. Normalizer 
subsystems are defined in a similar way (see \cite[\S I.5]{AKO}), but 
will not be needed here. 

\begin{Defi} \label{d:NF(Q)}
Let $\calf$ be a fusion system over a finite $p$-group $S$. For each 
$Q\le S$, the \emph{centralizer fusion subsystem} $C_\calf(Q)\le\calf$ is 
the fusion subsystem over $C_S(Q)$ defined by setting 
	\[ \Hom_{C_\calf(Q)}(P,R) = \bigl\{ \varphi|_P \,\big|\, 
	\varphi\in\homf(PQ,RQ),~ \varphi(P)\le R,~ \varphi|_Q=\Id_Q 
	\bigr\}. \]
\end{Defi}

Note that a subgroup $Q\le S$ is central in $\calf$ if and only 
if $C_\calf(Q)=\calf$. 

\begin{Thm}[{\cite[Theorem I.5.5]{AKO}}] \label{t:NF(Q)}
Let $\calf$ be a saturated fusion system over a finite $p$-group $S$, 
and fix $Q\le S$. Then $C_\calf(Q)$ is saturated if $Q$ is 
fully centralized in $\calf$. 
\end{Thm}

Weakly closed abelian subgroups play a central role in the paper, and the 
following lemma is of crucial importance when working with them.

\begin{Lem} \label{l:A-w.cl.}
Let $\calf$ be a saturated fusion system over a finite $p$-group $S$, and 
assume $A\le S$ is an abelian subgroup that is weakly closed in 
$\calf$. 
\begin{enuma} 

\item If $R\le S$ is fully normalized and $\calf$-conjugate to some $Q\le 
A$, then $R\le A$. 

\item For each $P,Q\le A$, each $\varphi\in\homf(P,Q)$ extends to some 
$\4\varphi\in\autf(A)$.

\end{enuma}
\end{Lem}

\begin{proof} \textbf{(a) }  Assume $Q\le A$ and $R\le S$ are 
$\calf$-conjugate, and $R$ is fully normalized in $\calf$. By the extension 
axiom, each $\psi\in\isof(Q,R)$ extends to some $\4\psi\in\homf(C_S(Q),S)$. 
Then $C_S(Q)\ge A$ since $A$ is abelian, $\4\psi(A)=A$ since $A$ is weakly 
closed in $\calf$, and so $R=\4\psi(Q)\le A$. 

\smallskip

\noindent\textbf{(b) } Assume $P,Q\le A$ and $\varphi\in\homf(P,Q)$, and 
choose $R\in Q^\calf$ that is fully centralized in $\calf$. Thus $R\le A$ 
by (a), and there is $\psi\in\isof(Q,R)$. By the extension axiom again, 
$\psi$ extends to $\5\psi\in\homf(A,S)$ and $\psi\varphi$ extends to 
$\5\varphi\in\homf(A,S)$, and $\5\psi(A)=A=\5\varphi(A)$ since $A$ is 
weakly closed. Then $\5\psi^{-1}\5\varphi\in\autf(A)$, and 
$(\5\psi^{-1}\5\varphi)|_P=\psi^{-1}(\psi\varphi)=\varphi$. 
\end{proof}

The proof of the next lemma gives another example of how the extension 
axiom can be used. 

\begin{Lem} \label{l:f.cent.+}
Let $\calf$ be a saturated fusion system over a finite $p$-group $S$, and 
let $A_0\le A_1\le S$ be a pair of abelian subgroups. If $A_0$ is fully 
centralized in $\calf$ and $A_1$ is fully centralized in $C_\calf(A_0)$, 
then $A_1$ is fully centralized in $\calf$.
\end{Lem}

\begin{proof} Choose $B_1\in A_1^\calf$ that is fully centralized in 
$\calf$, fix $\chi\in\isof(A_1,B_1)$, and set $B_0=\chi(A_0)$. By the 
extension axiom and since $A_0$ and $B_1$ are both fully centralized in 
$\calf$, there are $\varphi\in\homf(C_S(A_1),C_S(B_1))$ and 
$\psi\in\homf(C_S(B_0),C_S(A_0))$ such that $\varphi|_{A_1}=\chi$ and 
$\psi|_{B_0}=(\chi|_{A_0})^{-1}$. Since $C_S(B_1)\le C_S(B_0)$, the 
composite $\psi\varphi$ lies in $\Hom_{C_\calf(A_0)}(C_S(A_1),C_S(A_0))$.

Since $A_1$ is fully centralized in $C_\calf(A_0)$, 
	\[ \psi\varphi(C_S(A_1))=C_{C_S(A_0)}(\psi(B_1))
	=C_S(\psi(B_1))\ge\psi(C_S(B_1)), \] 
and hence $\varphi(C_S(A_1))\ge C_S(B_1)$. So $A_1$ is fully centralized in 
$\calf$ since $B_1$ is. 
\end{proof}

We will need to work with quotient fusion systems in Section \ref{s:Alp}, 
but only quotients by subgroups normal in the fusion system.

\begin{Defi} \label{d:F/Q}
Let $\calf$ be a fusion system, and assume $Q\nsg S$ is normal in $\calf$. 
Let $\calf/Q$ be the fusion system over $S/Q$ where for each $P,R\le S$ 
containing $Q$, we set
	\begin{multline*} 
	\Hom_{\calf/Q}(P/Q,R/Q) =\\ \bigl\{\varphi/Q\in\Hom(P/Q,R/Q) 
	\,\bigl|\, \varphi\in\homf(P,Q), ~ (\varphi/Q)(gQ)=\varphi(g)Q 
	~\forall\,g\in P \bigr\}.
	\end{multline*}
\end{Defi}

We refer to \cite[Proposition II.5.11]{Craven} for the proof that $\calf/Q$ 
is saturated whenever $\calf$ is. In fact, this definition and the 
saturation of $\calf/Q$ hold whenever $Q$ is weakly closed in $\calf$. This 
is not surprising, since we are looking only at morphisms in $\calf$ 
between subgroups containing $Q$, so that $\calf/Q=N_\calf(Q)/Q$.

\section{Some criteria for realizing representations}
\label{s:A<|F}

In this section, we state and prove our main technical results: the tools 
we later use to show that certain representations cannot be realized by any 
saturated fusion systems. Before doing that, we start by defining more 
formally what we mean by ``realizability''. 

\begin{Defi} \label{d:realize}
Fix a prime $p$, a finite abelian $p$-group $A$, and a subgroup 
$\Gamma\le\Aut(A)$. The pair $(\Gamma,A)$ is \emph{realized} by a saturated 
fusion system $\calf$ over a finite $p$-group $S$ if there is an abelian 
subgroup $B\le S$ such that $C_S(B)=B$ and $B\nnsg\calf$, and such that 
$(\autf(B),B)\cong(\Gamma,A)$. The pair $(\Gamma,A)$ is \emph{fusion 
realizable} if it is realized by some saturated fusion system over a finite 
$p$-group. 
\end{Defi}

If we drop the condition that $C_S(B)=B$, then it is easy to see that 
every pair $(\Gamma,A)$ can be realized by a saturated fusion system. For 
example, if $m>1$ is prime to $p$, then the fusion system $\calf$ of 
$(A\rtimes\Gamma)\wr C_m$ contains a subgroup isomorphic to $A$ with 
automizer isomorphic to $\Gamma$ which is not normal in $\calf$. Hence the 
importance of that condition in Definition \ref{d:realize}, although it 
seems possible that we would get similar results if it were replaced by the 
condition that $B$ be weakly closed. 

It is not yet clear to us whether the condition ``$B\nnsg\calf$'' is the 
optimal one to use in Definition \ref{d:realize}. It could be 
replaced by the slightly stronger condition that $\Omega_1(B)\nnsg\calf$, 
or by the even stronger condition that $O_p(\calf)=1$. In the cases dealt 
with in Theorems \ref{ThA} and \ref{ThB}, the result is the same 
independently of which definition we choose, but that probably does not 
hold in other situations. 

When applying Definition \ref{d:realize}, rather than assuming $(\Gamma,A)$ 
and $(\autf(B),B)$ are abstractly isomorphic, it will in practice be more 
convenient to say that $(\Gamma,A)$ is realized by a fusion system $\calf$ 
over $S$ if $S$ contains $A$ as a subgroup and $\autf(A)=\Gamma$. 

%%%%%%%%%%%%%%%%%%%%%%%%%%%%%%%

We are now ready to start developing tools for showing that certain pairs 
$(\Gamma,A)$ are not (weakly) fusion realizable. The starting point for all 
results in this section is the following proposition. It was inspired in 
part by \cite[Corollary 4]{Goldschmidt} and its proof, and also in part by 
arguments in \cite[\S\,1]{ONan}.

\begin{Prop} \label{p:not.str.cl.2}
Let $\calf$ be a saturated fusion system over a finite $p$-group $S$, and 
let $A\le S$ be an abelian subgroup. Assume $A\nnsg\calf$, and consider the 
sets 
	\begin{align*} 
	\scru &= \UU{\calf}{A} = \bigl\{ 1\ne U\le N_S(A) \,\big|\, U\nleq A,~ 
	\homf(U,A) \ne\emptyset \bigr\} \\
	\scrt &= \VV{\calf}{A} = \bigl\{ t\in N_S(A)\sminus A \,\big|\, 
	t^\calf\cap A\ne\emptyset \bigr\} 
	= \bigl\{ t\in N_S(A)\sminus A \,\big|\, \gen{t}\in\scru \bigr\} \\ 
	\scrw &= \WW{\calf}{A} = \bigl\{ (t,U,A_*) \,\big|\, 
	t\in\scrt,~ U\in\scru,~ C_A(t) \ge A_*\in(U\cap A)^\calf, ~ \\
	& \hskip100mm |UA/A|=|C_{A/A_*}(t)| \bigr\}.
	\end{align*}
Then $\scru\ne\emptyset$, $\scrt\ne\emptyset$, and 
$\scrw\ne\emptyset$, and the following hold. 
\begin{enuma} 

\item If $A$ is not weakly closed in $\calf$, there is $U\in 
A^\calf\sminus\{A\}$ such that $[U,A]\le U\cap A$, and such that 
$(t,U,U\cap A)\in\scrw$ for each $t\in U\sminus A$. 

\item If $A$ is weakly closed in $\calf$, then for each $t\in\scrt$, there 
are $U\in\scru$ and $A_*\le A$ such that $(t,U,A_*)\in\scrw$. 

\item If $A$ is weakly closed in $\calf$, then there is a subgroup 
$Z\le A$, fully centralized in $\calf$, such that $A\nnsg 
C_\calf(Z)$, and such that $U\cap A\le Z$ for each 
$U\in\UU{{C_\calf(Z)}}{A}$. In particular, $A_*=U\cap A$ for each 
$(t,U,A_*)\in\WW{{C_\calf(Z)}}{A}\subseteq \WW{\calf}{A}$. 

\end{enuma}
Thus in all cases, there are $t\in\scrt$ and $U\in\scru$ such that 
$(t,U,U\cap A)\in\scrw$.
\end{Prop}

\begin{proof} By Lemma \ref{l:s.cl.=>normal} and since $A\nnsg\calf$, $A$ 
is not strongly closed. So $\scru\ne\emptyset$ and $\scrt\ne\emptyset$ if 
$A\nsg S$, and we will show when proving (a) that this also holds if 
$A\nnsg S$. The last statement, and the claim $\scrw\ne\emptyset$, follow 
from (a) when $A$ is not weakly closed in $\calf$, and from (b) and (c) 
otherwise.

\smallskip

\noindent\textbf{(a) } If $A$ is not weakly closed in $\calf$, then there 
is $\varphi\in\homf(A,S)$ such that $\varphi(A)\ne A$. So by Theorem 
\ref{t:AFT} (Alperin's fusion theorem), there are $R\le S$ and 
$\alpha\in\autf(R)$ such that $A\le R$ and $\alpha(A)\ne A$. In the special 
case where $A\nnsg S$, we take $R=N_S(A)$, and set $\alpha=c_x^R$ for 
some $x\in N_S(R)\sminus R$. So in all cases, we can arrange that $A\nsg R$ 
and hence $\alpha(A)\le N_S(A)$. 

%%So in all cases, we can assume that $\alpha(A)\le N_S(A)$. 

Set $U=\alpha(A)\in\scru$ and $A_*=U\cap A$. Then $[A,U]\le A_*$ since $A$ 
and $U$ are both normal in $R$. So for each $t\in U\sminus A\subseteq\scrt$, 
we have $A_*\le C_A(U)\le C_A(t)$ and $|UA/A| = |U/A_*| = |A/A_*| = 
|C_{A/A_*}(t)|$, proving that $(t,U,A_*)\in\scrw$.

\smallskip

\noindent\textbf{(b) } Assume $A$ is weakly closed in $\calf$ (in 
particular, $A\nsg S$). Fix $t\in\scrt$, and let $\scru_t$ be the set of 
all $U\in\scru$ such that $t\in U$. Choose $V\in\scru_t$ such that $|V\cap 
A|$ is maximal among all $|U\cap A|$ for $U\in\scru_t$. Set $A_*=V\cap A$ 
and $U_2^*=N_A(A_*\gen{t})$. Then $A_*\gen{t}\cap A\le V\cap A=A_*$, and so 
	\beqq U_2^*/A_* = \bigl\{x\in A\,\big|\,[x,t]\in 
	A_*\bigr\} \big/ A_* = C_{A/A_*}(t) \ne 1, \label{e:U2*/A*} \eeqq
where $C_{A/A_*}(t)\ne1$ since $A/A_*$ and $t$ both have $p$-power order. 

Choose $W\in (A_*\gen{t})^\calf$ such that $W$ is fully normalized in 
$\calf$. Then $W\le A$ by Lemma \ref{l:A-w.cl.}(a) and since $A$ is weakly 
closed. Let $\varphi\in\homf(N_S(A_*\gen{t}),S)$ be such that 
$\varphi(A_*\gen{t})=W$ (see Proposition \ref{p:Hom(NSP,S)}). 

Set $U=\varphi(U_2^*)$ and $U_1^*=\varphi^{-1}(U\cap A)$. Then 
	\[ \varphi(A_*) \le \varphi(U_2^*)\cap A = U\cap A = 
	\varphi(U_1^*), \]
so $A_*\le U_1^*\le U_2^*\le A$. Also, $U_1^*\gen{t}\in\scru_t$ since 
$\varphi(U_1^*\gen{t})=(U\cap A)\gen{\varphi(t)}\le A$, and hence 
	\[ |U_1^*| \le |U_1^*\gen{t}\cap A| \le |V\cap A| = |A_*| \]
by the maximality assumption on $V$. Thus $U_1^*=A_*<U_2^*$ where the 
strict inclusion holds by \eqref{e:U2*/A*}, and 
$A_*=U_1^*\in(U\cap A)^\calf$. 

Now, $U\cap A=\varphi(U_1^*)<\varphi(U_2^*)=U$, so $U\nleq A$. 
Since $U=\varphi(U_2^*)$ where $U_2^*\le A$, this shows that 
$U\in\scru$. Also, $U\cap A=\varphi(A_*)$, and so $UA/A\cong U/(U\cap 
A)\cong U_2^*/A_*=C_{A/A_*}(t)$. Thus $(t,U,A_*)\in\scrw$. 

\smallskip

\noindent\textbf{(c) } Again assume $A$ is weakly closed in $\calf$, and 
let $Z$ be maximal among all subgroups of $A$ fully centralized in 
$\calf$ such that $A\nnsg C_\calf(Z)$. Set $\calf_0=C_\calf(Z)$ 
and $S_0=C_S(Z)$ for short. Recall that $\calf_0$ is saturated 
since $Z$ is fully centralized in $\calf$ (Theorem \ref{t:NF(Q)}). 

Fix $U\in\UU{{\calf_0}}{A}$, choose a morphism $\varphi\in\homf[0](U,A)$, 
and set $A_*=U\cap A$. We must show that $A_*\le Z$. Since 
$UZ\in\UU{{\calf_0}}{A}$, we can assume $U\ge Z$. 

Choose $B_*\in(A_*)^{\calf_0}$ that is fully normalized in $\calf_0$. Then 
$B_*\le A$ by Lemma \ref{l:A-w.cl.}(a) and since $A$ is weakly closed. By 
Proposition \ref{p:Hom(NSP,S)}, there is 
$\chi\in\Hom_{\calf_0}(N_{S_0}(A_*),S_0)$ such that $\chi(A_*)=B_*$. Then 
$\chi(A)=A$ since $A$ is weakly closed, so 
$\chi\varphi(\chi|_U)^{-1}\in\homf[0](\chi(U),A)$ where 
$Z\le\chi(U)\nleq A$ and $B_*=\chi(U\cap A)=\chi(U)\cap A$, and where 
$B_*\le Z$ if and only if $A_*\le Z$. Upon replacing $U$ by $\chi(U)$ 
and $\varphi$ by $\chi\varphi(\chi|_U)^{-1}$, we are now reduced to showing 
that $A_*\le Z$ when $A_*=U\cap A$ is fully centralized in $\calf_0$, and 
hence in $\calf$ by Lemma \ref{l:f.cent.+}. 

By Lemma \ref{l:A-w.cl.}(b), there is an automorphism 
$\alpha\in\autf[0](A)$ such that $\alpha|_{A_*}=\varphi|_{A_*}$, hence such 
that $\alpha^{-1}\varphi\in\Hom_{C_\calf(A_*)}(U,A)$. Since $U\nleq A$, 
this implies that $A\nnsg C_\calf(A_*)$, and so $A_*=Z$ by the maximality 
assumption on $Z$. 

In particular, for each $(t,U,A_*)\in\WW{{\calf_0}}{A}$, since 
$U\cap A\le Z$ and $A_*\in(U\cap A)^{\calf_0}$, we have $U\cap 
A=A_*\le Z$.
\end{proof}

We now reformulate the criteria in Proposition \ref{p:not.str.cl.2} in 
terms of $A$ and $\autf(A)$ only; i.e., in terms that do not involve the 
fusion system $\calf$ or its Sylow group $S$.

\begin{Defi} \label{d:K&R}
Fix a finite abelian $p$-group $A$ and a $p$-subgroup $T\le\Aut(A)$. Set 
	\begin{align*} 
	\hRR[+]{T}{A} &= \bigl\{ (\tau,B,A_*) \,\big|\, 
	\tau\in T^\#,~  
	B\le T,~ \textup{$\gen{\tau}$ and $B$ isomorphic to subgroups of $A$,}~ 
	\\[-1mm] 
	&\hskip60mm 
	A_*\le C_A(\gen{B,\tau}),~ |B|\ge|C_{A/A_*}(\tau)| \bigr\} 
	\\[1mm]
	\hRR{T}{A} &= \bigl\{ (\tau,B,A_*)\in\hRR[+]{T}{A} 
	\,\big|\, |B|=|C_{A/A_*}(\tau)| \bigr\}.
	\end{align*}
Let $\RR{T}{A}$ be the largest subset 
$\calr\subseteq\hRR{T}{A}$ that satisfies the condition
	\beq \textup{for each $(\tau,B,A_*)\in\calr$ and each $\tau_1\in 
	B^\#$, there is $(\tau_1,B_1,A_{*1})\in\calr$.} \tag{$*$} \eeq
Similarly, let $\RR[+]{T}{A}$ be the largest subset $\calr\subseteq\hRR[+]{T}{A}$ 
that satisfies ($*$).
\end{Defi}

If $\scrr_1$ and $\scrr_2$ are two subsets of $\hRR{T}{A}$ or of 
$\hRR[+]{T}{A}$ that satisfy ($*$), then their union also satisfies ($*$). 
So there are unique largest subsets $\RR{T}{A}\subseteq\RR[+]{T}{A}$ that 
satisfy the condition.

\begin{Prop} \label{p:not.str.cl.3a}
Let $\calf$ be a saturated fusion system over a finite $p$-group $S$, and 
assume $A\le S$ is an abelian subgroup such that $C_S(A)=A$ and 
$A\nnsg\calf$. Then $\RR{\Aut_S(A)}{A}\ne\emptyset$, and hence 
$\RR[+]{\Aut_S(A)}{A}\ne\emptyset$. More precisely, the following hold, 
where $T=\Aut_S(A)$:
\begin{enuma} 

\item In all cases, if $(t,U,A_*)\in\WW{\calf}{A}$ is such that $U\cap 
A=A_*$, then $(c_t^A,\Aut_U(A),A_*)\in\hRR{T}{A}$. 

\item If $A$ is not weakly closed in $\calf$, then there is a subgroup $U\in 
A^\calf\sminus\{A\}$ such that $(c_t^A,\Aut_U(A),A\cap 
U)\in\RR{T}{A}$ for each $t\in U\sminus A$.

\item If $A$ is weakly closed in $\calf$, then there is a subgroup 
$Z\le A$ fully centralized in $\calf$ such that $A\nnsg 
C_\calf(Z)$, and such that for each 
$t\in\VV{{C_\calf(Z)}}{A}$, there is $U\in\UU{{C_\calf(Z)}}{A}$ 
such that 
	\[ U\cap A\le Z \qquad\textup{and}\qquad 
	(c_t^A,\Aut_U(A),U\cap A)\in\RR{C_T(Z)}{A} 
	\subseteq \RR{T}{A}. \]

\end{enuma}
\end{Prop}

\begin{proof} Let $\calf$ be a saturated fusion system over a finite 
$p$-group $S$ as above. Thus $A\le S$ is such that $C_S(A)=A$ and 
$A\nnsg\calf$. Once we have proven points (a), (b), and (c), it will then 
follow immediately that $\RR{T}{A}\ne\emptyset$.

\smallskip

\noindent\textbf{(a) } Fix $(t,U,A_*)\in\WW{\calf}{A}$ such that 
$A_*=U\cap A$, and set $\tau=c_t^A\in T$ and $B=\Aut_U(A)\le T$. Then 
$A_*=U\cap A\le C_A(B)$. Also, 
by definition of $\WW{\calf}{A}$, we have $A_*\le C_A(t)=C_A(\tau)$ and 
$|UA/A|=|C_{A/A_*}(t)|=|C_{A/A_*}(\tau)|$. 

By definition of $\VV{\calf}{A}$ and $\UU{\calf}{A}$, the subgroups 
$\gen{\tau}$ and $B$ are both isomorphic to subgroups of $A$. So to prove 
that $(\tau,B,A_*)\in\hRR{T}{A}$, it remains only to show that 
$|UA/A|=|B|$. But $C_S(A)=A$ by assumption, so $|B|=|\Aut_U(A)|=|UA/A|$.

\smallskip

\noindent\textbf{(b) } If $A$ is not weakly closed in $\calf$, then by 
Proposition \ref{p:not.str.cl.2}(a), there is $U\in A^\calf\sminus\{A\}$ 
such that $[U,A]\le U\cap A$, and such that $(t,U,U\cap A) 
\in\WW{\calf}{A}$ for each $t\in U\sminus A$. Thus $(c_t^A,\Aut_U(A),U\cap 
A)\in\hRR{\calf}{A}$ for each $t\in U\sminus A$ by (a). 

Now set $\scrr = \{ (\tau,\Aut_U(A),U\cap A) \,|\, \tau\in 
B^\# \}\subseteq\hRR{\calf}{A}$. Then $\scrr$ satisfies condition ($*$) in 
Definition \ref{d:K&R}, so $\RR{T}{A}\supseteq\scrr\ne\emptyset$. 

\smallskip

\noindent\textbf{(c) } Assume $A$ is weakly closed in $\calf$, and let 
$Z\le A$ be as in Proposition \ref{p:not.str.cl.2}(c). Thus $Z$ 
is fully centralized in $\calf$, $A\nnsg C_\calf(Z)$, and $U\cap A\le Z$ 
for each $U\in\scru_{C_\calf(Z)}(A)$.

Let $\scrt=\VV{{C_\calf(Z)}}{A}\ne\emptyset$, 
$\scru=\UU{{C_\calf(Z)}}{A}\ne\emptyset$, and 
$\scrw=\WW{{C_\calf(Z)}}{A}\ne\emptyset$ be as in Proposition 
\ref{p:not.str.cl.2}, and set 
	\[ \scrr = \bigl\{ (c_t^A,\Aut_U(A),U\cap 
	A) \,\big|\, t\in\scrt,~ U\in\scru,~ (t,U,A_*)\in\scrw \bigr\}, \]
where $A_*\in(U\cap A)^{C_\calf(Z)}$ and hence $A_*=U\cap A$ since $U\cap 
A\le Z$. By (a), $\scrr\subseteq\hRR{C_T(Z)}{A}$. By Proposition 
\ref{p:not.str.cl.2}(b,c), for each $t\in\scrt$, there is $U\in\scru$ such 
that $(t,U,U\cap A)\in\scrw$. So $\scrr\ne\emptyset$, and condition ($*$) 
in Definition \ref{d:K&R} holds for the pair $\scrr$. Thus $\scrr \subseteq 
\RR{C_T(Z)}{A} \subseteq \RR{T}{A}$. 
\end{proof}

The next proposition is our main reason for defining $\RR[+]{T}{A}$.

\begin{Prop} \label{p:not.str.cl.3c}
Fix a finite abelian $p$-group $A$ and a $p$-subgroup $T\le\Aut(A)$. 
Let $A_1<A_2\le A$ be $T$-invariant subgroups such that $T$ 
acts faithfully on $A_2/A_1$. If $\RR[+]{T}{A}\ne\emptyset$, then 
$\RR[+]{T}{A_2/A_1}\ne\emptyset$. More precisely, 
	\[ \RR[+]{T}{A_2/A_1} \supseteq 
	\bigl\{(\tau,B,(A_*A_1\cap A_2)/A_1) \,\big|\, 
	(\tau,B,A_*)\in\RR[+]{T}{A} \bigr\}. \]
\end{Prop}

\begin{proof} Assume $A_1<A_2\le A$ are as above. If 
$(\tau,B,A_*)\in\hRR[+]{T}{A}$, then 
	\[ |C_{A_2/(A_*A_1\cap A_2)}(\tau)|\le|C_{A_2/(A_*\cap A_2)}(\tau)|
	= |C_{A_2A_*/A_*}(\tau)|\le |C_{A/A_*}(\tau)| \le |B|: \]
the first inequality by Lemma \ref{l:C(A/A0)G} and the second by inclusion. 
So $(\tau,B,(A_*A_1\cap A_2)/A_1)\in \hRR[+]{T}{A_2/A_1}$. 

In particular, if $\scrr$ satisfies condition ($*$) in Definition 
\ref{d:K&R} for the pair $(T,A)$, then $\scrr'$ satisfies ($*$) 
for $(T,A_2/A_1)$ where 
	\beq \scrr' = \bigl\{(\tau,B,(A_*A_1\cap A_2)/A_1) \,\big|\, 
	(\tau,B,A_*)\in\scrr \bigr\}. \qedhere \eeq
\end{proof}

It remains to find some strong necessary conditions on 
$A$ and $T$ for the set $\RR{T}{A}$ or $\RR[+]{T}{A}$ to be nonempty.

\begin{Prop} \label{p:not.str.cl.3b}
Fix a finite abelian $p$-group $A$ and a subgroup $T\le\Aut(A)$. 
Then for each $(\tau,B,A_*)\in\hRR{T}{A}$, 
	\beqq 
	|B| = \frac{|A|}{|A_*[\tau,A]|} 
	\quad\textup{and}\quad 
	\frac{|B|}{|C_A(\tau)\cap[\tau,A]|} =
	\frac{|C_A(\tau)[\tau,A]|}{|A_*[\tau,A]|}, 
	\label{e:B1} \eeqq
while for each $(\tau,B,A_*)\in\hRR[+]{T}{A}$, 
	\beqq 
	|B| \ge \frac{|A|}{|A_*[\tau,A]|} 
	\quad\textup{and}\quad 
	\frac{|B|}{|C_A(\tau)\cap[\tau,A]|} \ge 
	\frac{|C_A(\tau)[\tau,A]|}{|A_*[\tau,A]|} \ge 1. 
	\label{e:B1+} \eeqq
In particular, for each $(\tau,B,A_*)\in\hRR[+]{T}{A}$, 
	\beqq |B| \ge |C_A(\tau)\cap[\tau,A]|, \label{e:B2} \eeqq
and $|B|\ge|[\tau,A]|$ if $p=2$ and $A$ is elementary abelian. 
\end{Prop}

\begin{proof} For each $\tau\in T^\#$, let $\varphi_{\tau}\in\End(A)$ 
be the map $\varphi_{\tau}(a)=[\tau,a]$. For each $A_*\le C_A(\tau)$, we 
have $C_A(\tau)=\Ker(\varphi_{\tau})$ and 
$C_{A/A_*}(\tau)=\varphi_{\tau}^{-1}(A_*)/A_*$, and hence 
	\beqq \begin{split} 
	|C_{A/A_*}(\tau)| 
	= \frac{|C_A(\tau)|\cdot|A_*\cap[\tau,A]|}{|A_*|} 
	&= \frac{|C_A(\tau)|\cdot|[\tau,A]|}{|A_*[\tau,A]|} 
	= \dfrac{|A|}{|A_*[\tau,A]|} \\[2mm]
	&\hskip20mm=\dfrac{|C_A(\tau)[\tau,A]| \cdot|C_A(\tau)\cap[\tau,A]|}
	{|A_*[\tau,A]|}. 
	\label{e:B3} \end{split} \eeqq
Since $|B|\ge|C_{A/A_*}(\tau)|$ for each 
$(\tau,B,A_*)\in\hRR[+]{T}{A}$ with equality if 
$(\tau,B,A_*)\in\hRR{T}{A}$, points \eqref{e:B1} and \eqref{e:B1+} 
follow immediately from \eqref{e:B3} (and since $A_*\le C_A(\tau)$). 
Inequality \eqref{e:B2} follows from \eqref{e:B1+}, and the last statement 
holds since $[\tau,A]\le C_A(\tau)$ if $p=2$ and $A$ is elementary abelian.
\end{proof}

The following corollary describes one easy consequence of the above 
results. 

\begin{Cor} \label{c:not.str.cl.}
Fix a finite abelian $p$-group $A$ and a $p$-subgroup $T\le\Aut(A)$ such 
that $\RR[+]{T}{A}\ne\emptyset$. Then there is $B_0\le T$, isomorphic to a 
subgroup of $A$, such that $|B_0|\ge|C_A(\tau)\cap[\tau,A]|$ for each 
$\tau\in B_0^\#$.
\end{Cor}

\begin{proof} Assume $\RR[+]{T}{A}\ne\emptyset$. Choose 
$(\tau_0,B_0,A_{*0})\in\RR[+]{T}{A}$ such that 
$|C_A(\tau_0)\cap[\tau_0,A]|$ is the largest possible. By condition 
($*$) in Definition \ref{d:K&R}, for each $\tau\in B_0^\#$, there is 
$(\tau,B,A_*)\in\RR[+]{T}{A}$, and hence 
	\[ |C_A(\tau)\cap[\tau,A]| \le |C_A(\tau_0)\cap[\tau_0,A]| 
	\le |B_0|, \]
where the second inequality holds by \eqref{e:B2}. 
\end{proof}

We can think of the inequality $|B_0|\ge|C_A(\tau)\cap[\tau,A]|$ in 
Corollary \ref{c:not.str.cl.} as a generalization of the condition 
$|Z(S)\cap[S,S]|=p$ in \cite[Lemma 2.3(b)]{indp1}. More precisely, when 
$A$ has index $p$ in $S$ and $S$ is 
nonabelian, the corollary says that $|C_A(\tau)\cap[A,\tau]|=p$ 
for $\tau\in S\sminus A$, and hence that $|Z(S)\cap[S,S]|=p$.

We next look at the case where $A$ is elementary abelian. For 
$\tau\in\End(A)$, we regard $A$ as an $\F_p[X]$-module, and let the 
``Jordan blocks'' for $\tau$ be the factors under some decomposition of $A$ 
as a product of indecomposable submodules. As usual, by ``nontrivial Jordan 
blocks'' we really mean ``Jordan blocks with nontrivial action''. 

The following notation will be used when reformulating Corollary 
\ref{c:not.str.cl.} in terms of Jordan blocks. 

\begin{Not} \label{d:JV(x)}
Let $A$ be an elementary abelian $p$-group, and let $\tau\in\Aut(A)$ be an 
automorphism of $p$-power order. Set 
$\scrj_A(\tau)=\rk(C_A(\tau)\cap[\tau,A])$: the number of nontrivial Jordan 
blocks for the action of $\tau$ on $A$. 
\end{Not}

In these terms, Corollary \ref{c:not.str.cl.} takes the following form 
when $A$ is elementary abelian:

\begin{Cor} \label{c:not.str.cl.2}
Assume $\Gamma$ is a finite group such that $\Gamma=O^{p'}(\Gamma)$, and 
let $A$ be a finite faithful $\F_p\Gamma$-module. Assume there is a 
saturated fusion system $\calf$ over a finite $p$-group $S$ that realizes 
$(\Gamma,A)$ as in Definition \ref{d:realize}. Then there are $m\ge1$ and 
an elementary abelian $p$-subgroup $B\le\Gamma$ of rank $m$ such 
that $\scrj_A(\tau)\le m$ for each $\tau\in B^\#$. 
\end{Cor}

\begin{proof} Since $\calf$ realizes $(\Gamma,A)$, we can arrange that 
$A\le S$, $A\nnsg\calf$, and $\autf(A)=\Gamma$. Set 
$T=\Aut_S(A)\in\sylp\Gamma$. Then $\RR[+]{T}A\ne\emptyset$ by Proposition 
\ref{p:not.str.cl.3a}. So by Corollary \ref{c:not.str.cl.}, there is an 
elementary abelian $p$-subgroup $B\le \Gamma$ such that 
$|B|\ge|C_{A}(\tau)\cap[\tau,A]|$ for all $\tau\in B^\#$. Thus $\rk(B)\ge 
\rk\bigl(C_{A}(\tau)\cap[A,\tau]\bigr) = \scrj_{A}(\tau)$ for each $\tau\in 
B^\#$. 
\end{proof}

The special case of fusion realizability when $|T|=p$ was already handled 
in the earlier papers \cite{indp1} and \cite{indp2}. We state the main 
conditions found in those papers: 

\begin{Lem} \label{l:indp}
Fix a finite abelian $p$-group $A$ and subgroups $\Gamma\le\Aut(A)$ and 
$T\in\sylp\Gamma$, and assume that $|T|=p$ and $|[T,A]|>p$. If $(\Gamma,A)$ 
is fusion realizable, then 
	\[ |C_A(T)\cap[T,A]|=p \qquad\textup{and}\qquad 
	|N_\Gamma(T)/C_\Gamma(T)|=p-1. \]
\end{Lem}

\begin{proof} The first equality is just a special case of Corollary 
\ref{c:not.str.cl.}. 

To see the second equality, assume that $(\Gamma,A)$ is realized by the 
fusion system $\calf$ over $S\ge A$. In particular, we can assume that 
$\Aut_S(A)=T$, and so $|N_S(A)/A|=|T|=p$. Also, $|A/C_A(T)|=|[T,A]|>p$ by 
assumption, so $A$ is the only abelian subgroup of index $p$ in $N_S(A)$. 
Hence $A\nsg S$, since otherwise $A\ne\9xA\le N_S(A)$ for $x\in 
N_S(N_S(A))\sminus N_S(A)$. 

By Theorem \ref{t:AFT} and since $A\nnsg\calf$ (recall $\calf$ realizes 
$(\Gamma,A)$), there must be some 
$\calf$-essential subgroup $P\le S$ other than $A$, and by \cite[Lemma 
2.2(a)]{indp2}, $P\in\calh\cup\calb$ where the classes $\calh$ and $\calb$ 
of subgroups of $S$ are defined in \cite[Notation 2.1]{indp2}. 
By \cite[Lemma 2.6(a)]{indp2} (and in terms of Notation 2.4 in \cite{indp2}), 
we have $\mu(\autf^{(P)}(S))=\Delta_t$ for $t=0$ or $-1$, and from the 
definition of $\mu$ it then follows that $\Aut_\Gamma(T)=\Aut(T)$ and hence 
has order $p-1$. 
\end{proof}

\section{Representations of Mathieu groups} 
\label{s:Mathieu}

We next look at representations of the Mathieu groups $M_n$ and their 
central extensions. The main theorem is stated for an arbitrary prime $p$, 
but we focus attention mostly on the cases $p=2,3$, since the others 
follow from Lemma \ref{l:indp} and results in \cite{indp2}.

We will apply Corollary \ref{c:not.str.cl.2} in most cases, using Lemma 
\ref{l:J(x)ge...} and the character tables in \cite{modatlas} to find lower 
bounds for $\scrj_A(x)$ when $|x|=2$ or $3$. The notation $\2x$ and $\3x$ 
refers to the classes as named in the Atlas \cite{atlas} and in 
\cite{modatlas}. In the following lemma, we restrict attention to $M_{12}$ 
and $M_{24}$ since they are the only Mathieu groups with more than one 
conjugacy class of elements of order $2$ or $3$.

\begin{Lem} \label{l:M12,M24}
Assume $\Gamma\cong M_{12}$ or $M_{24}$. Then 
\begin{enuma} 

\item each element of order $2$ in $\Gamma$ is contained in some 
$H_1\le\Gamma$ with $H_1\cong D_{10}$; and

\item each element of order $3$ in $\Gamma$ is contained in some 
$H_3\le\Gamma$ with $H_3\cong A_4$, and with elements of order $2$ in class 
\2a (if $\Gamma\cong M_{12}$) or \2b (if $\Gamma\cong M_{24}$). 

\end{enuma}
\end{Lem}

\begin{proof} Let $n=12,24$ be such that $\Gamma\cong M_n$, and let $X$ be 
a $5$-fold transitive $\Gamma$-set of order $n$. In each case, $\Gamma$ has 
two classes of elements of order $2$ and two classes of elements of order 
$3$, and they are distinguished by whether they act on $X$ freely or with 
fixed points as described in Table \ref{tbl:2A2B}. The outer automorphism 
of $M_{12}$ sends each of these classes to itself, and so the inclusion of 
$\Aut(M_{12})$ into $M_{24}$ sends distinct classes to distinct classes. It 
thus suffices to prove the lemma when $\Gamma\cong M_{12}$. 
	\begin{Table}[ht]
	\[ \renewcommand{\arraystretch}{1.3} 
	\renewcommand{\arraycolsep}{4mm}
	\begin{array}{c|cccc} 
	\Gamma & \2a & \2b & \3a & \3b \\\hline
	M_{12} & 2^6 & 2^4\cdot1^4 & 3^3\cdot1^3 & 3^4 \\
	M_{24} & 2^8\cdot1^8 & 2^{12} & 3^6\cdot1^6 & 3^8 
	\end{array} \]
	\caption{The table lists the number of orbits in the action on $X$ 
	by each element of order $2$ or $3$ in $\Gamma$. Thus, for example, 
	a $\2b$-element in $M_{12}$ acts with four orbits of 
	length two and four fixed points.} 
	\label{tbl:2A2B}
	\end{Table}

\noindent\textbf{(a) } Fix an element $g\in\2a$. By \cite[p. 41]{GL}, 
$C_\Gamma(g)\cong C_2\times\Sigma_5$, and the second factor must permute 
faithfully the six orbits under the action of $g$. Fix $N\le C_\Gamma(g)$ 
of order $5$, and let $h\in C_\Gamma(g)\sminus\gen{g}$ be such that 
$N\gen{h}\cong D_{10}$. Then $N\gen{gh}\cong D_{10}$, and we will be done 
upon showing that $h$ and $gh$ lie in different classes. 

Set $X_0=C_X(N)$, a subset of order $2$ whose elements are exchanged by 
$g$, and set $X_1=X\sminus X_0$. Of the two elements $h$ and $gh$, one 
fixes the two points in $X_0$ and the other exchanges them, and we can 
assume that $h$ fixes them. Hence $C_X(h)\ne\emptyset$, so $h\in\2b$. Also, 
$C_X(gh)\subseteq X_1$, and since $gh$ permutes freely 
four of the five $\gen{g}$-orbits in $X_1$, we have $|C_X(gh)|\le2$. Since 
no involution in $M_{12}$ acts with exactly two fixed points, this shows 
that $gh\in\2a$, finishing the proof of (a).

\noindent\textbf{(b) } Now fix an element $g\in\3b$. Then 
$C_\Gamma(g)\cong C_3\times A_4$ by \cite[p. 41]{GL}. Set 
$N=O_2(C_\Gamma(g))\cong E_4$. The group 
$C_\Gamma(g)/\gen{g}\cong A_4$ acts faithfully on the set of four orbits of 
$g$, so the elements of order $2$ in $N$ all act freely on $X$ 
and hence lie in class \2a. 

Fix $h\in C_\Gamma(g)$ such that $N\gen{h}\cong A_4$. Then $N\gen{gh}$ and 
$N\gen{g^2h}$ are also isomorphic to $A_4$. Also $h$ permutes freely three 
of the four $\gen{g}$-orbits in $X$, and the fourth orbit is fixed by 
exactly one of the elements $h$, $gh$, or $g^2h$. So one of these three 
elements lies in class $\3a$,and the other two in class $\3b$. 
\end{proof}

There are two special cases that we will need to consider separately. The 
statement and proof of the following proposition are based on notation 
set up in Appendices \ref{s:Todd-F2} and \ref{s:3M22}.

\begin{Prop} \label{p:M22-23}
Assume $p=2$.
\begin{enuma} 

\item If $\Gamma\cong M_{22}$ or $M_{23}$ and $A$ is the Golay module (dual 
Todd module) for $\Gamma$, then $\RR[+]{T}{A}=\emptyset$ for 
$T\in\syl2\Gamma$.

\item If $\Gamma\cong3M_{22}$ and $A$ is the $6$-dimensional simple 
$\F_4\Gamma$-module, then $\RR[+]{T}{A}=\emptyset$ for $T\in\syl2\Gamma$. 

\end{enuma}
\end{Prop}

\begin{proof} In the first part of the proof, we consider cases (a) and (b) 
together. Assume the proposition is not true, and fix a triple 
$(\tau,B,A_*)\in\RR[+]{T}{A}$. Thus $\tau\in T$ has order $2$, $B\le T$ is 
an elementary abelian $2$-subgroup, and $A_*\le C_A(\gen{B,\tau})$ is such 
that $|B|\ge|C_{A/A_*}(\tau)|$. By Proposition \ref{p:not.str.cl.3b} and 
since $[\tau,A]\le C_A(\tau)$, we have 
	\beqq |B| \ge \bigl|[\tau,A]\bigr| \cdot \bigl|C_A(\tau) / 
	A_*[\tau,A] \bigr| \ge |[\tau,A]|. \label{e:R-ineq} \eeqq

Since $\Gamma\cong M_{22}$, $M_{23}$, or $3M_{22}$ has only one conjugacy 
class of involution, we have $|[\tau,A]|=2^4$: by Lemma \ref{l:hexad.grp} 
in case (a), and by Lemma \ref{l:3M22b}(b) in case (b). Thus $|B|\ge2^4$, with 
equality since $\rk_2(\Gamma)=4$ in all cases. So the inequalities in 
\eqref{e:R-ineq} are equalities, $C_A(\tau) = A_*[\tau,A]$, and hence 
	\beqq \rk(C_A(B))\ge \rk(A_*)\ge\rk(C_A(\tau)/[\tau,A]) 
	= \rk(A)-2\cdot\rk([\tau,A]) = \rk(A)-8. \label{e:rkA-8} \eeqq

\smallskip

\noindent\textbf{(a) } Assume $T<\Gamma$ and $A$ are as in Notation 
\ref{not:Gamma-action} and \ref{n:M22-23}. Since $H_1$ and $H_2$ are the only 
subgroups of $T\in\syl2{\Gamma_0}$ isomorphic to $E_{16}$ by Lemma 
\ref{l:hexad.grp}, $B$ must be equal to one of them. 
Since $C_A(H_1)$ has rank $1$ by Lemma \ref{l:hexad.grp} again, and 
$\rk(C_A(B))\ge\rk(A)-8\ge2$ by \eqref{e:rkA-8}, we have $B=H_2$. 

%%By Lemma \ref{l:hexad.grp} again, $C_A(H_1)$ has rank $1$ and 
%%$C_A(H_2)$ has rank $4$, so $B=H_2$. 

By condition ($*$) in Definition \ref{d:K&R}, each element of 
$B^\#$ can appear as the first component in an element of 
$\RR[+]{T}{A}$. So we can assume that $(\tau,B,A_*)$ was chosen such that 
$\tau=\trsh1$ (and still $B=H_2$). Hence by Tables \ref{tbl:[x,a]} and 
\ref{tbl:[ch1,x]},
	\[ \grfh2+C_{56}\in C_A(\trsh1) = A_*[\trsh1,A]
	\le C_A(H_2)[\trsh1,A] = \gen{C_{12},C_{13},C_{14},C_{15},\grfh1}, \] 
a contradiction. We conclude that $\RR[+]{T}{A}=\emptyset$. 

\smallskip

\iffalse
We can assume $T<\Gamma$ and $A$ are as in Notation 
\ref{n:3M22} and \ref{n:3M22b}.  Assume $\RR[+]{T}{A}\ne\emptyset$, and fix 
an element $(\tau,B,A_*)\in\RR[+]{T}{A}$. Thus $\tau\in T$ has order $2$, 
$B\le T$ is an elementary abelian $2$-group, and $A_*\le 
C_{A}(B\gen{\tau})$ is such that $\rk(B)\ge\rk([\tau,A])$. By Proposition 
\ref{p:not.str.cl.3b} and since $[\tau,A]\le C_A(\tau)$, 
	\beqq |B| \ge |[\tau,A]|\cdot \bigl[C_{A}(\tau) : A_*[\tau,A] \bigr] 
	\ge |[\tau,A]|. \label{e:3M22c} \eeqq
Since $\Gamma$ has only one conjugacy class of involutions, $\tau$ is 
$\Gamma$-conjugate to $\mu_{10}$, and hence 
$\rk([\tau,A])=2\cdot\dim_{\F_4}([\mu_{10},A])=4$ (Lemma 
\ref{l:3M22b}(b)). Since $\rk(B)\le\rk_2(\Gamma)=4$, we have $\rk(B)=4$ and 
$C_{A}(\tau)=A_*[\tau,A]$ by \eqref{e:3M22c}. So  
	\[ \rk(C_A(B)) \ge \rk(A_*)\ge \rk(C_A(\tau))-\rk([\tau,A]) = 
	\rk(C_A(\mu_{10}))-\rk([\mu_{10},A]) = 4. \]
where the last equality holds by Lemma \ref{l:3M22b}(b).
\fi

\noindent\textbf{(b) } Now assume $T<\Gamma$ and $A$ are as in Notation 
\ref{n:3M22} and \ref{n:3M22b}. By Lemma 
\ref{l:3M22b}(a), $P_1$ and $P_2$ are the only subgroups of $T$ isomorphic 
to $E_{16}$. Since $\rk(C_A(P_2))=2\cdot\dim_{\F_4}(C_A(P_2))=2$ by Lemma 
\ref{l:3M22b}(b), while $\rk(C_A(B))\ge4$ by \eqref{e:rkA-8}, we have 
$B=P_1$. By condition ($*$) in Definition \ref{d:K&R}, we can assume that 
the triple $(\tau,P_1,A_*)$ was chosen so that $\tau=\mu_{10}$. But then
	\[ \gen{e_1,e_2,e_3,e_4} = C_A(\mu_{10}) = 
	A_*[\mu_{10},A]\le C_A(P_1)[\mu_{10},A] 
	= \gen{e_1,e_2,e_3} \]
by Lemma \ref{l:3M22b}(b), a contradiction. 
\end{proof}

We now apply Corollary \ref{c:not.str.cl.2} and Lemma \ref{l:J(x)ge...}, 
together with Proposition \ref{p:M22-23}, to determine the realizability of 
$\F_p\Gamma$-modules when $O^{p'}(\Gamma)$ is a central extension of a 
Mathieu group. The following is a restatement of Theorem \ref{ThA}. 

\iffalse
most $\F_p\Gamma$-modules cannot be realized in the sense of Definition 
\ref{d:realize}.
\fi

\begin{Thm} \label{t:M11-24}
Fix a prime $p$ and a finite group $\Gamma$, and set 
$\Gamma_0=O^{p'}(\Gamma)$. Assume that $\Gamma_0$ is quasisimple, 
and that $\Gamma_0/Z(\Gamma_0)$ is one of the Mathieu groups. Let $A$ be an 
$\F_p\Gamma$-module such that $(\Gamma,A)$ is fusion realizable, and 
set $A_0=[\Gamma_0,A]/C_{[\Gamma_0,A]}(\Gamma_0)$. Then either 
\begin{enuma} 

\item $p=2$, $\Gamma\cong M_{22}$ or $M_{23}$, and $A_0$ is the Todd 
module for $\Gamma$; or 

\item $p=2$, $\Gamma\cong M_{24}$, and $A_0$ is the Todd 
module or Golay module for $\Gamma$; or 

\item $p=3$, $\Gamma\cong M_{11}$, $M_{11}\times C_2$, or $2M_{12}$, and 
$A_0$ is the Todd module or Golay module for $\Gamma_0$; or

\item $p=11$, $\Gamma_0\cong2M_{12}$ or $2M_{22}$, 
$\Gamma/Z(\Gamma_0)\cong\Aut(M_{12})\times C_5$ or $\Aut(M_{22})\times C_5$, 
and $A_0$ is a $10$-dimensional simple $\F_{11}\Gamma$-module. 

\end{enuma}
\end{Thm}

\begin{proof} Let $n\in\{11,12,22,23,24\}$ be such that 
$\Gamma_0/Z(\Gamma_0)\cong M_n$. Fix $T\in\sylp{\Gamma}=\sylp{\Gamma_0}$. 
We will frequently refer to Tables \ref{tbl:JA(t)-2} and \ref{tbl:JA(t)-3} 
for our lower bounds on $\scrj_A(\tau)$ for $|\tau|=p$, and they in turn 
are based on Lemmas \ref{l:M12,M24} and \ref{l:J(x)ge...} and the character 
tables in the Atlas of Brauer characters \cite{modatlas}.

\smallskip

\noindent\textbf{Case 1: } If $p>3$, then $|T|=p$ in all cases. 
So by Lemma \ref{l:indp}, we have $|N_\Gamma(T)/C_\Gamma(T)|=p-1$ and 
$|C_A(T)\cap[T,A]|=p$. In the terminology of \cite{indp2}, this translates 
to saying that $\Gamma\in\scrg_p^\wedge$ and $A$ is minimally active, and 
so the result follows from \cite[Proposition 7.1]{indp2}.

\smallskip

\noindent\textbf{Case 2: } Assume $p=2$. By Table \ref{tbl:JA(t)-2}, for 
$\tau\in\Gamma$ of order $2$, we have $\scrj_{A_0}(\tau)>\rk_2(\Gamma)$ (and 
hence $\RR[+]{T}{A_0}=\emptyset$) for each nontrivial simple 
$\F_2\Gamma_0$-module $A_0$, except when $\Gamma_0\cong M_{22}$, $M_{23}$, 
or $M_{24}$ and $A_0$ is the Todd module or Golay module. 
	\begin{Table}[ht]
	\[ \renewcommand{\arraystretch}{1.2} 
	\begin{array}{cccc|c} 
	\Gamma_0 & \rk_2(\Gamma_0) & \dim(A_0) & \tau\in & \scrj_{A_0}(\tau) 
	\\\hline
	M_{11} & 2 & >1 & \2a & 
	\scrj_{A_0}(\2a)\ge\frac25(\chi_{A_0}(1)-\chi_{A_0}(\0{5a}))\ge4 \\
	M_{12} & 3 & >1 & \2a,\2b & 
	\scrj_{A_0}(\2x)\ge\frac25(\chi_{A_0}(1)-\chi_{A_0}(\0{5a}))\ge4 \\
	M_{22} & 4 & >10 & \2a & 
	\scrj_{A_0}(\2a)\ge\frac25(\chi_{A_0}(1)-\chi_{A_0}(\0{5a}))\ge8 \\
	3M_{22} & 4 & >12 & \2a & 
	\scrj_{A_0}(\2a)\ge\frac25(\chi_{A_0}(1)-\chi_{A_0}(\0{5a}))\ge6 \\
	M_{23} & 4 & >11 & \2a & 
	\scrj_{A_0}(\2a)\ge\frac25(\chi_{A_0}(1)-\chi_{A_0}(\0{5a}))\ge8 \\
	M_{24} & 6 & >11 & \2a,\2b & 
	\scrj_{A_0}(\2x)\ge\frac25(\chi_{A_0}(1)-\chi_{A_0}(\0{5a}))\ge8 
	\end{array} \]
	\caption{In all cases, $A_0$ is an $\F_2\Gamma$-module such that 
	$C_{A_0}(\Gamma)=0$ and $[\Gamma,A_0]=A_0$, and the characters are 
	taken with respect to $\F_2$. The bounds for $\scrj_{A_0}(\tau)$ 
	all follow from Lemmas \ref{l:M12,M24}(a) and 
	\ref{l:J(x)ge...}(a).} 
	\label{tbl:JA(t)-2}
	\end{Table}

Thus if $Z(\Gamma_0)$ has odd order, then either $n\ge22$ and $A_0$ is the 
Todd module or Golay module for $\Gamma$, or $\Gamma_0\cong3M_{22}$ and 
$A_0$ is the $6$-dimensional $\F_4\Gamma_0$-module. In these cases, 
$\RR[+]{T}{A_0}=\emptyset$ by Proposition \ref{p:M22-23}, and so they are 
impossible by Propositions \ref{p:not.str.cl.3a} and \ref{p:not.str.cl.3c}. 

\iffalse
In the latter case, 
$\RR[+]{T}{A_0}=\emptyset$ by Proposition \ref{p:3M22}, while if $\Gamma_0\cong 
M_{22}$ or $M_{23}$ and $A_0$ is its Golay module, then 
$\RR[+]{T}{A_0}=\emptyset$ by Proposition \ref{p:dTodd}. So these last 
cases are impossible by Propositions \ref{p:not.str.cl.3a} and 
\ref{p:not.str.cl.3c}. 
\fi

It remains to consider the cases where $Z(\Gamma)$ has even order. Assume 
first that $\Gamma_0\cong2M_{12}$. Then $\rk_2(\Gamma)=4$, and 
$\scrj_{A_0}(\tau)\ge4$ for each $\F_2[\Gamma/Z(\Gamma)]$-module $A_0$ with 
nontrivial action by Table \ref{tbl:JA(t)-2}. By the last statement in 
Lemma \ref{l:pG-rep} (applied with $A$ in the role of $V$), for each 
elementary abelian 2-subgroup $B\le G$ of rank $4$, since $Z(\Gamma)\le B$, 
there is $\tau\in B$ of order $2$ such that $\scrj_{A}(\tau)\ge5$. So 
Corollary \ref{c:not.str.cl.2} again applies to show that $(\Gamma,A)$ is 
not fusion realizable.

Now assume that $\Gamma_0/Z(\Gamma_0)\cong M_{22}$, and let $Z\le 
Z(\Gamma)$ be the Sylow 2-subgroup. Thus $|Z|=2$ or $4$, and 
$\rk_2(\Gamma_0)\le5$.  By Table \ref{tbl:JA(t)-2} and since 
$\scrj_{A_0}(\tau)\le\rk_2(\Gamma_0)$, either $\Gamma_0/Z\cong M_{22}$ and 
$A_0$ is its Todd module or its dual, or $\Gamma_0/Z\cong3M_{22}$ and $A_0$ 
is the $6$-dimensional $\F_4\Gamma/Z$-module. By Lemma \ref{l:pG-rep}(b) 
and since $\Gamma$ acts faithfully on $A$, there must be indecomposable 
extensions of $A_0$ by $\F_2$ and of $\F_2$ by $A_0$. Thus 
$H^1(\Gamma/Z;A_0)\ne0$ and $H^1(\Gamma/Z;A_0^*)\ne0$ (where $A_0^*$ is the 
dual module), contradicting \cite[Lemma 6.1]{MS}. We conclude that no such 
faithful $\F_2\Gamma$-modules exist. 

\smallskip

\noindent\textbf{Case 3: } Assume $p=3$. We claim that 
$\scrj_{A_0}(\tau)>\rk_3(\Gamma_0)$ (and hence $(\Gamma,A)$ is not fusion 
realizable) in all cases except when $\Gamma_0\cong M_{11}$ or $2M_{12}$ 
and $A_0$ is the Todd module for $\Gamma_0$ or its dual. This follows from 
Table \ref{tbl:JA(t)-3} except when $\Gamma_0\cong M_{11}$, 
$\dim(A_0)=10$, and $A_0\oplus\F_3$ is the $11$-dimensional permutation 
module. But in that case, $\scrj_{A_0}(\tau)=3$ whenever $|\tau|=3$ since 
$\tau$ acts on an $11$-set with three free orbits. 

	\begin{Table}[ht]
	\[ \renewcommand{\arraystretch}{1.2} 
	\begin{array}{cccc|c} 
	\Gamma & \rk_3(\Gamma) & \dim(A_0) & \tau\in & \scrj_{A_0}(\tau) 
	\\\hline
	M_{11} & 2 & \ge10~\textup{(*)} & \3a & \scrj_{A_0}(\3a) 
	\ge\frac14(\chi_{A_0}(1)-\chi_{A_0}(\0{4a}))\ge \frac52  \\
	M_{12} & 2 & >1 & \3a,\3b & 
	\scrj_{A_0}(\3x)\ge\frac14(\chi_{A_0}(1)-\chi_{A_0}(\2a))\ge3  \\
	2M_{12} & 2 & >6 & \3a,\3b & \scrj_{A_0}(\3x)
	\ge\frac14(\chi_{A_0}(1)-\chi_{A_0}(\2a))\ge\frac52 \\
	M_{22} & 2 & >1 & \3a & \scrj_{A_0}(\3a) 
	\ge\frac14(\chi_{A_0}(1)-\chi_{A_0}(\2a))\ge 4 \\
	2M_{22} & 2 & >1 & \3a & \scrj_{A_0}(\3a) 
	\ge\frac14(\chi_{A_0}(1)-\chi_{A_0}(\2a))\ge 4 \\
	M_{23} & 2 & >1 & \3a & \scrj_{A_0}(\3a) 
	\ge\frac14(\chi_{A_0}(1)-\chi_{A_0}(\2a))\ge 4 \\
	M_{24} & 2 & >1 & \3a,\3b & 
	\scrj_{A_0}(\3x)\ge\frac14(\chi_{A_0}(1)-\chi_{A_0}(\2b))\ge6 
	\end{array} \]
	\caption{In all cases, $A_0$ is an $\F_3\Gamma$-module such that 
	$C_{A_0}(\Gamma)=0$ and $[\Gamma,A_0]=A_0$, and the characters are 
	taken with respect to $\F_3$. Thus when $\Gamma\cong2M_{22}$, the 
	character values for the simple 10-dimensional 
	$\4\F_3\Gamma$-module are doubled here since it can only be 
	realized over $\F_9$. When $\Gamma\cong M_{11}$, the bounds for 
	$\scrj_{A_0}(\tau)$ apply only when $A_0$ is not the 
	$10$-dimensional permutation module. The bounds for 
	$\scrj_{A_0}(\tau)$ all follow from Lemma \ref{l:M12,M24} and Lemma 
	\ref{l:J(x)ge...}(c), except when $\Gamma\cong M_{11}$ or $2M_{12}$ 
	where \ref{l:J(x)ge...}(d) is used.} 
	\label{tbl:JA(t)-3}
	\end{Table}

Finally, if $\Gamma_0\cong M_{11}$ or $2M_{12}$ and $A$ is the Todd 
module or its dual, then $A$ is absolutely irreducible by \cite[Lemmas 4.2 
and 5.2]{O-todd}, and hence $\Gamma\cong M_{11}$, $M_{11}\times C_2$, or 
$2M_{12}$.
\end{proof}

\section{Alperin's 2-groups of normal rank 3}
\label{s:Alp}

As an example of how the results in Section \ref{s:A<|F} can be applied 
when the abelian $p$-subgroup $A<S$ is not elementary abelian, we next 
look at some $2$-groups first studied by Alperin \cite{Alperin} and O'Nan 
\cite{ONan}. These are groups $A\nsg S$ where $A\cong C_{2^n}\times 
C_{2^n}\times C_{2^n}$ and $S/A\cong D_8$, with presentation given in Table 
\ref{tbl:D8onA}. They are characterized by Alperin \cite[Theorem 
1]{Alperin} as the Sylow $2$-subgroups of groups $G$ with normal subgroup 
$E\cong E_8$, such that $O(G)=1$, $\Aut_G(E)=\Aut(E)$ and all 
involutions in $C_G(E)$ lie in $E$. Our goal is to show how results in 
Section \ref{s:A<|F} can be applied to prove in 
the context of fusion systems a theorem of O'Nan's, by showing that $A$ 
is normal in all saturated fusion systems over $S$ \cite[Lemma 1.10]{ONan}.
	\begin{Table}[ht]
	\[ \renewcommand{\arraystretch}{1.4}
	\renewcommand{\arraycolsep}{4mm}
	\begin{array}{c|cccc} 
	v & v^t & v^s & v^{s^2} & v^{st} \\\hline
	v_1 & v_3^{-1} & v_2 & v_3 & v_2^{-1} \\
	v_2 & v_2^{-1} & v_3 & v_1v_2^{-1}v_3 & v_1^{-1} \\
	v_3 & v_1^{-1} & v_1v_2^{-1}v_3 & v_1 & v_1^{-1}v_2v_3^{-1} 
	\end{array}
	\]
	\caption{Let $S=A\gen{s,t}$, where $A=\gen{v_1,v_2,v_3}\cong 
	C_{2^n}\times C_{2^n}\times C_{2^n}$, 
	the elements $s$ and $t$ act on $A$ as described in the table, 
	and also $t^2=1$ and $s^4\in\gen{v_1v_3}$. Set 
	$T=\Aut_S(A)=\gen{c_s,c_t}\cong D_8$.} 
	\label{tbl:D8onA}
	\end{Table}

Before considering the groups $A\nsg S$ directly, we must first handle the 
following, simpler case (compare with \cite[Lemma 1.7]{ONan}). 

\begin{Lem} \label{l:ON}
Fix $n\ge2$, and let $\5S=\gen{v,w,\sigma}$ be a group of order $2^{2n+2}$, 
where $\5A=\gen{v,w}\cong C_{2^n}\times C_{2^n}$, and 
$\5S=\5A\rtimes\gen{\sigma}$ where $\sigma^4=1$, $v^\sigma=w$, and 
$w^\sigma=v^{-1}$. Then $\5A$ is normal in every saturated fusion system 
over $\5S$. 
\end{Lem}

\begin{proof} Assume otherwise: assume $\calf$ is a saturated fusion system 
over $\5S$ for which $\5A\nnsg\calf$. Thus some element $t\in\5S\sminus\5A$ 
is $\calf$-conjugate to an element of $\5A$, and upon replacing $t$ by 
$t^2$ if necessary, we can arrange that $t\in\sigma^2\5A$. Since 
$|C_{\5A}(\sigma)|=2$ and $|C_{\5A}(\sigma^2)|=4$, each abelian subgroup of 
$\5S$ not contained in $\5A$ has order at most $8$, and hence $\5A$ is 
weakly closed in $\calf$. 

By Proposition \ref{p:not.str.cl.2}(b,c) and since $\5A$ is weakly closed in 
$\calf$, there is $U\le \5S$ $\calf$-conjugate to a subgroup of $\5A$ such 
that $(t,U,U\cap \5A)\in\WW{\calf}{\5A}$. In particular, 
$|U\5A/\5A|=|C_{\5A/(U\cap \5A)}(t)|$. 

Since conjugation by $t$ sends each element of $\5A$ to its inverse, $U\cap 
\5A\le C_{\5A}(t)=\Omega_1(\5A)$, and hence $C_{\5A/(U\cap 
\5A)}(t)=\Omega_1(\5A/(U\cap \5A))$ has order $4$. Thus $|U\5A/\5A|=4$, and 
so there is $u\in U$ such that $u\in\sigma \5A$. 

We claim that for each $U^*\in U^\calf$, either $U^*\5A=\5S$ or $U^*\le 
\5A$. Assume otherwise: then $U^*\5A=\5A\gen{\sigma^2}$. So $U^*\cap \5A\le 
C_{\5A}(\sigma^2)=\Omega_1(\5A)$, and $U^*$ is elementary abelian since 
each element of $\sigma^2 \5A$ has order $2$. Since $U\cong U^*$ is not 
elementary abelian (recall $|u|=4$), this is impossible.

By Theorem \ref{t:AFT} (Alperin's fusion theorem), there is a subgroup 
$R\le \5S$, together with an automorphism $\alpha\in\autf(R)$ and subgroups 
$A_1$ and $U_1=\alpha(A_1)$, such that $A_1,U_1\in U^\calf$, $A_1\le \5A$, 
and $U_1\nleq \5A$. We just saw that this implies $U_1\5A=\5S$. 
So $\5A\cap R$ contains a cyclic subgroup of 
order $4$ and is normalized by $\sigma$. Hence 
$R\ge\gen{v^{2^{n-1}},(vw)^{2^{n-2}}}$, and so $[R,R]\ge\Omega_1(\5A)$. Since 
$\alpha$ sends some element of $\Omega_1(\5A)$ to an element in the coset 
$\sigma^2\5A\nsubseteq[R,R]$, this is impossible. 
\end{proof}

Lemma \ref{l:ON} can also be proven using the transfer for $\calf$ (see, 
e.g., \cite[\S I.8]{AKO}) to show that no element $x^2$, for $x\in\sigma 
\5A$, can be in the focal subgroup of $\calf$. Such an argument would be 
closer to that used by O'Nan in the proof of \cite[Lemma 1.7]{ONan}, but we 
wanted to apply the tools used elsewhere in this paper.

We now return to the groups $A\nsg S$ defined by the presentation in Table 
\ref{tbl:D8onA}. We first check that when $n\ge2$, $A$ is weakly closed in 
every saturated fusion system over $S$:

\begin{Lem}[{\cite[Lemma 1.5]{ONan}}] \label{l:Alp-w.cl.}
Let $S=A\gen{s,t}$ be an extension of the form described in Table 
\ref{tbl:D8onA}, where $n\ge2$. Then $A$ is the only abelian subgroup of 
index $8$ in $S$, and hence is weakly closed in every saturated fusion 
system over $S$.
\end{Lem}

\begin{proof} This follows immediately from the centralizers listed in 
Table \ref{tbl:CA(H)}, since if $A_1<S$ were abelian of index $8$ and 
$A_1\ne A$, then for $x\in A_1\sminus A$, the subgroup $C_A(x)\ge A\cap 
A_1$ would have index at most $4$ in $A$.
	\begin{Table}[ht]
	\[ \renewcommand{\arraystretch}{1.4}
	\begin{array}{c|cccccc} 
	H & \gen{t} & \gen{s^2} & \gen{st} & \gen{s} & \gen{s^2,t} 
	& \gen{s^2,st} \\\hline
	C_A(H) & \gen{v_1v_3^{-1},v_2^\gee} & \gen{v_1v_3,v_2^\gee 
	v_3^\gee} & \gen{v_1v_2^{-1},v_2^\gee v_3^\gee} & \gen{v_1v_3} 
	& \gen{v_1^\delta v_2^\gee v_3^{-\delta}} & \gen{v_1^\gee v_3^\gee, 
	v_2^\gee v_3^\gee} \\
	{}[H,A] & \gen{v_1v_3,v_2^2} & \gen{v_1v_3^{-1},v_1^2v_2^{-2}} 
	& \gen{v_1v_2,v_2^2v_3^{-2}} & \gen{v_1v_2^{-1},v_2v_3^{-1}} &  & 
	\end{array}
	\]
	\caption{Centralizers and commutators involving some of the abelian 
	subgroups $H\le\gen{s,t}$. Here, $\gee=2^{n-1}$ and 
	$\delta=2^{n-2}$.} 
	\label{tbl:CA(H)}
	\end{Table}
\end{proof}

The arguments used in the proof of the following theorem are essentially 
the same as O'Nan's (when proving Lemma 1.10 in \cite{ONan}), but 
repackaged with the help of Proposition \ref{p:not.str.cl.3a} and the 
properties of the sets $\RR{T}{A}$.

\begin{Thm}[{\cite[Lemma 1.10]{ONan}}] \label{t:Alp}
Let $S=A\gen{s,t}$ be an extension of the form described in Table 
\ref{tbl:D8onA}, where $n\ge3$. Then $A$ is normal in every saturated 
fusion system $\calf$ over $S$. 
\end{Thm}

\begin{proof} Assume otherwise: assume $\calf$ is such that $A\nnsg\calf$. By 
Proposition \ref{p:not.str.cl.3a}(c) and since $A$ is weakly 
closed in $\calf$ by Lemma \ref{l:Alp-w.cl.}, there is a subgroup 
$Z\le A$ fully centralized in $\calf$ such that $A\nnsg 
C_\calf(Z)$, and such that for each 
$u\in\VV{C_\calf(Z)}{A}$ there is $U\in\UU{C_\calf(Z)}{A}$ such 
that $U\cap A\le Z$ and $(c_u^A,\Aut_U(A),U\cap A)\in\RR{T}{A}$. Set 
$\tau=c_u^A$; we can assume that $|\tau|=2$. Set $B=\Aut_U(A)$ and 
$A_*=U\cap A$. 

By Table \ref{tbl:CA(H)}, we have $|C_A(\tau)\cap[\tau,A]|=4$. So 
$|B|\ge4$ by inequality \eqref{e:B2} in Proposition \ref{p:not.str.cl.3b}, 
with equality since $T\cong D_8$ has no abelian subgroups of order $8$. 
Hence 
	\beqq C_A(B)[\tau,A] \ge A_*[\tau,A] = C_A(\tau)[\tau,A], 
	\label{e:Alp1} \eeqq
where the equality follows from \eqref{e:B1} in Proposition 
\ref{p:not.str.cl.3b}.

Since $|B|=4$, we have $c_{s^2}\in B$. So we can choose $u\in s^2A$ with 
$u\in\VV{C_\calf(Z)}{A}$ (thus $C_\calf(Z)$-conjugate to an 
element of $A$), and hence $\tau=c_u^A=c_{s^2}$. By Table 
\ref{tbl:CA(H)},
	\[ [\tau,A]=\gen{v_1v_3^{-1},v_1^2v_2^{-2}} 
	\qquad\textup{and}\qquad 
	C_A(\tau)[\tau,A]=\gen{v_1v_3,v_1^2,v_2^2}. \]
So by Table \ref{tbl:CA(H)}, point \eqref{e:Alp1} fails when 
$B=\gen{s^2,t}$ or $\gen{s^2,st}$, and holds only when $B=\gen{s}$ and 
$A_*=C_A(s)=\gen{v_1v_3}$. Since $A_*\le Z\le C_A(B)$ by assumption, 
we have $Z=\gen{v_1v_3}$.

Set $\5\calf=C_\calf(Z)/Z$, $\5A=A/Z$, and 
$\5S=C_S(Z)/Z$ (see Definition \ref{d:F/Q}). Then $A\nnsg 
C_\calf(Z)$ by assumption, hence is not strongly closed by Lemma 
\ref{l:s.cl.=>normal}, and so $A/Z$ is not strongly closed in 
$C_\calf(Z)/Z$. Thus $\5A\nnsg\5\calf$. Let $v,w,\sigma\in \5S$ be 
the classes (modulo $Z$) of $v_1,v_2,s\in S$. Then $\5A\nsg\5S$ are as 
in Lemma \ref{l:ON}, so $\5A\nsg\5\calf$ by that lemma, giving a 
contradiction. 
\end{proof}

\appendix

\section{Some lemmas in representation theory}
\label{s:JV(x)}

Recall Notation \ref{d:JV(x)}: when $V$ is an elementary abelian 
$p$-group and $\tau\in\Aut(V)$ has order $p$, we set 
	\[ \scrj_V(\tau)=\rk\bigl(C_V(\tau)\cap[\tau,V]\bigr): \]
the number of nontrivial Jordan blocks under the action of $\tau$ on $V$. 
We derive here some formulas that give lower bounds for these functions in 
terms of Brauer characters.

The first lemma gives, in certain cases, lower bounds for $\scrj_V(x)$ in 
terms of the modular character of $V$. When $q$ is a prime and $q\nmid n$, 
we let $\ord_q(n)$ denote the order of $n$ in the group $\F_q^\times$. 

\begin{Lem} \label{l:J(x)ge...}
Fix a prime $p$, an elementary abelian $p$-group $V$, and an element 
$x\in\Aut(V)$ of order $p$. Let $\chi=\chi_V$ be the modular character of 
$V$ as an $\F_p\Aut(V)$-module.
\begin{enuma} 
\item Assume $p=2$, and let $q$ be an odd prime such that $\ord_q(2)=q-1$. 
Let $a\in\Aut(V)$ be such that $|a|=q$ and $\gen{a,x}\cong 
D_{2q}$. Then 
	\[ \scrj_V(x) \ge 
	\tfrac{q-1}{2q}\bigl(\chi_V(1)-\chi_V(a)\bigr). \]

\item Let $q$ be a prime such that $p\mid(q-1)$, and let $a\in\Aut(V)$ be 
such that $|a|=q$ and $\gen{a,x}$ is nonabelian of order $pq$. Then 
	\[ \scrj_V(x) \ge \frac1{pq}\sum_{i=1}^{q-1}
	\bigl(\chi_V(1)-\chi_V(a^i)\bigr). \]

\item Assume $p=3$, and let $a\in\Aut(V)$ be such that $\gen{a,x}\cong 
A_4$ and $|a|=2$. Then 
	\[ \scrj_V(x) \ge \tfrac14\bigl(\chi_V(1)-\chi_V(a)\bigr). \]

\item Assume $p=3$, and let $a\in\Aut(V)$ be such that 
$\gen{a,x}\cong2A_4$ and $|a|=4$. Then 
	\[ \scrj_V(x) \ge \tfrac14\bigl(\chi_V(1)-\chi_V(a)\bigr). \] 
\end{enuma}
\end{Lem}

\begin{proof} \noindent\textbf{(b) } Since $\gen{a,x}$ is 
nonabelian of order $pq$, where $p\mid(q-1)$ and $|a|=q$, we have 
	\[ \dim(V/C_V(a)) = \chi_V(1) - \frac1q\sum_{i=0}^{q-1}\chi_V(a^i) 
	= \frac1q\sum_{i=1}^{q-1}(\chi_V(1)-\chi_V(a^i)). \]
The action of $x$ on $\4\F_p\otimes_{\F_p}(V/C_V(a))$ 
permutes freely the eigenspaces for $a$, corresponding to the primitive 
$q$-th roots of unity in $\4\F_p$. So all Jordan blocks 
for this action have length $p$, and the same holds for Jordan blocks 
for the action of $x$ on $V/C_V(a)$. So 
$\scrj_V(x)\ge\scrj_{V/C_V(a)}(x)=\frac1p\dim(V/C_V(a))$.

\smallskip

\noindent\textbf{(a) } Since $|a|=q$ and $\ord_q(2)=q-1$, we have 
$\chi_V(a^i)=\chi_V(a)$ for all $i$ prime to $q$. So this is a special case 
of (b). 

\smallskip

\noindent\textbf{(c) } Let $b\in\gen{a,x}\cong A_4$ be such that 
$\gen{a,b}\cong E_4$. Since $a$, $b$, and $ab$ are permuted cyclically by 
$x$, they all have the same character. Hence each of the three 
nontrivial irreducible characters for $\gen{a,b}\cong E_4$ 
appears with multiplicity
	\[ n = \tfrac13\dim(V/C_V(\gen{a,b})) = \tfrac13\bigl( \chi_V(1) - 
	\tfrac14(\chi_V(1)+3\chi_V(a)) \bigr) = \tfrac14(\chi_V(1)-\chi_V(a)) 
	. \]
Since $x$ permutes those three characters cyclically, we have 
$\scrj_V(x)\ge n$.

\smallskip

\noindent\textbf{(d) } Set $H=\gen{a,x}\cong2A_4$ where $|a|=4$, and set 
$z=a^2\in Z(H)$. Then $V=V_+\oplus V_-$ as $\F_3H$-modules, where $V_\pm$ 
are the eigenspaces for the action of $z$, and it suffices to prove the 
claim when $V=V_+$ or $V=V_-$. The case $V=V_+$ was shown in (c).

Now assume $V=V_-$, and set $m=\dim(V)=\chi_V(1)$ and $H_0=O_2(H)\cong 
Q_8$. Let $W$ be the (unique) irreducible $2$-dimensional $\F_3H_0$-module. 
Then $V|_{H_0}\cong W^{m/2}$, and $\Hom_{\F_3H_0}(W,V)\cong\F_3^{m/2}$ 
since $\End_{\F_3H_0}(W)\cong\F_3$. So there are $\frac12(3^{m/2}-1)$ 
submodules of $V|_{H_0}$ isomorphic to $W$, they are permuted by 
$\gen{x}\cong C_3$, and hence there is at least one $2$-dimensional 
$\F_3H$-submodule $W_1\le V$. By applying the same argument to $V/W_1$ and 
then iterating, we get a sequence $0=W_0<W_1<\cdots<W_k=V$ of 
$\F_3H$-submodules such that $\dim(W_i/W_{i-1})=2$ for each $1\le i\le k$. 
Then $\dim(C_{W_i/W_{i-1}}(x))=1$ for each $i$, so $\dim(C_{V}(x))\le m/2$, 
and $\dim([x,V])\ge m/2$. Each nontrivial Jordan block in $V$ has dimension 
$2$ or $3$, and intersects with $[x,V]$ with dimension $1$ or $2$, 
respectively. Thus 
	\[ \scrj_{V}(x) \ge \tfrac12\dim([x,V]) \ge \tfrac14m = 
	\tfrac14\chi_V(1) = \tfrac14(\chi_V(1)-\chi_V(a)), \]
the last equality since $\chi_V(a)=0$ (recall $a^2=z$ acts on $V$ via 
$-\Id$). 
\end{proof}

The next lemma is needed to handle $\F_p\Gamma$-modules in certain cases 
where $O_p(\Gamma)\ne1$.

\begin{Lem} \label{l:pG-rep}
Fix a prime $p$, a finite group $G$ such that $O^{p}(G)=G$, and a subgroup 
$1\ne Z\le Z(G)$ of $p$-power order. Set $\4G=G/Z$. Let $V$ be a faithful 
indecomposable $\F_pG$-module. Then either 
\begin{enuma} 
\item among the composition factors of $V$, there are at least two simple 
$\F_pG$-modules with nontrivial action of $G$; or 

\item there are submodules $0\ne V_0<V_1<V$ such that $G$ acts trivially on 
$V_0$ and on $V/V_1$, the $\F_p\4G$-module $V_1/V_0$ is simple, and 
$V_1$ and $V/V_0$ have trivial $Z$-action and are indecomposable 
$\F_p\4G$-modules.

\end{enuma}
Furthermore, in the situation of (b), for each $g\in G\sminus Z$, 
we have $\rk([h,V_1/V_0])=\rk([h,V])$ for at most one element $h\in gZ$. 
Thus if $p=2$ and $|g|=2$, there is $h\in gZ$ of order $2$ such that 
$\scrj_V(h)>\scrj_{V_1/V_0}(h)$.
\end{Lem}

\begin{proof} Assume (a) does not hold. Thus all but one of the composition 
factors in $V$ have trivial $G$-action, and there are $\F_pG$-submodules 
$V_0<V_1\le V$ such that $V_1/V_0$ is simple (hence $Z$ acts trivially) 
and all composition factors of $V_0$ and of $V/V_1$ are trivial. Since 
$G=O^p(G)$ is generated by $p'$-elements, it acts trivially on $V_0$ and on 
$V/V_1$.

Let $W\le V_1$ be the submodule generated by the $[g,V_1]$ for all 
$p'$-elements $g\in G$. For each such $g$, $[g,V_1]\cap V_0\le[g,V_1]\cap 
C_{V_1}(g)=0$ since $g$ acts trivially on $V_0$, so projection onto 
$V_1/V_0$ sends $[g,V_1]$ injectively, and $Z$ acts trivially on $[g,V_1]$ 
since it acts trivially on $V_1/V_0$. Thus $[Z,W]=0$, and $V_1=W+V_0$ since 
$V_1/V_0$ is simple and $W\nleq V_0$. So $Z$ acts trivially on $V_1$. 

By a similar argument, $Z$ acts trivially on the dual $(V/V_0)^*$, and 
hence acts trivially on $V/V_0$. Since $Z$ acts nontrivially on $V$, we 
have $V_1<V$ and $V_0\ne0$.

Assume $V_1$ is not indecomposable. Thus $V_1=W_0\oplus W_1$, where $W_0$ 
and $W_1$ are nontrivial $\F_p\4G$-submodules of $V_1$ and $W_0\le V_0$. 
The action of $G$ on $V/W_1$ is trivial (an extension of $W_0$ by $V/V_1$), 
so $[G,V]\le W_1$, and $W_0$ splits off as a direct summand of $V$, 
contradicting the assumption that $V$ be indecomposable. Thus $V_1$ is 
indecomposable as an $\F_p\4G$-module, and a similar argument involving the 
dual module $V^*$ shows that $V/V_0$ is also indecomposable, finishing 
the proof of (b).

Now fix $g\in G\sminus Z$, and assume that $h_1,h_2\in gZ$ are distinct 
elements such that $\rk([h_i,V])=\rk([h_i,V_1/V_0])$ for $i=1,2$. Set 
$z=h_1^{-1}h_2\in Z^\#$. Since $G$ acts faithfully on $V$ by 
assumption, there is some $a_0\in V$ such that $[z,a_0]\ne0$. By (b), we 
have $a_0\notin V_1$ and $[z,a_0]\in V_0$. 

Set $h=h_1$ for short, so that $h_2=zh$. Then $[h,V_1/V_0]=[hz,V_1/V_0]$, 
so $\rk([h,V])=\rk([h,V_1/V_0])=\rk([hz,V])$, and hence 
$[h,V]=[h,V_1]=[hz,V]$ and $[h,V_1]\cap V_0=0$. In particular, $[h,a_0]$ 
and $[hz,a_0]$ are both in $[h,V_1]$. Also,
	\[ [hz,a_0] = z(h(a_0)-a_0)+(z(a_0)-a_0) = z([h,a_0])+[z,a_0], \]
so $0\ne[z,a_0]\in[h,V_1]\cap V_0$, a contradiction.

The last statement now follows since if $p=2$ and $|h|=2$, then 
$\scrj_V(h)=\rk([h,V])$ and $\scrj_{V_1/V_0}(h)=\rk([h,V_1/V_0])$. 
\end{proof}

The following example shows one way to construct examples of modules of the 
type described in Lemma \ref{l:pG-rep}(b).

\begin{Ex} \label{ex:pG-ext}
Fix a prime $p$, a finite group $G$ such that $O^{p}(G)=G$, and a subgroup 
$1\ne Z\le Z(G)$ of $p$-power order. Choose $k\ge1$ such that $Z$ has 
exponent at most $p^k$. Let $H<G$ be such that no nontrivial normal 
subgroup of $G$ is contained in $H$. Set $\5V=\Z/p^k(G/H)$: the free 
$\Z/p^k$-module with basis the set $G/H$ of left cosets. Regard $\5V$ as a 
left $\Z/p^kG$-module, set $V_2=C_Z(\5V)$, and let $V\le\5V$ be such that 
$V/V_2=C_{\5V/V_2}(G)$. Set $V_0=C_V(G)=C_{V_2}(G)$ and $V_1=[G,V_2]V_0$. 
Then $V$ is a $\Z/p^kG$-module on which $G$ acts faithfully. Also, $G$ acts 
trivially on $V_0$ and on $V/V_1$, and $Z$ acts trivially on $V_1$ and on 
$V/V_0$. 

If, furthermore, $V_1<V_2$ (equivalently, if $p\bmid|G/HZ|$), then there is 
a $\Z/p^kG$-submodule $V'<V$ such that $V'>V_1$, $G$ acts faithfully on 
$V'$, and $V'/V_1\cong V/V_2$. 
\end{Ex}

\begin{proof} Set 
	\[ \sigma_G = \sum_{gH\in G/H}gH \in C_{\5V}(G)=V_0 
	\qquad\textup{and}\qquad 
	\sigma_Z= \sum_{z\in Z}zH \in C_{\5V}(Z)=V_2. \] 
Note that $Z\cap H=1$ since it is normal in $G$ and contained in $H$. 

Since no nontrivial normal subgroup of $G$ is contained in $H$, the group 
$G$ acts faithfully on $\5V$ and $G/Z$ acts faithfully on $V_2$. So $G$ 
acts faithfully on $V$ if $Z$ does. 

Fix an element $1\ne z\in Z$; we will show that $[z,V]\ne0$. Let $Z_0<Z$ 
and $x\in Z\sminus Z_0$ be such that $Z=Z_0\times\gen{x}$ and $z\notin 
Z_0$, and set $p^\ell=|x|$ (thus $\ell\le k$). Choose $\lambda\in\Z/p^k$ of 
order $p^\ell$, let $g_1,\dots,g_m\in G$ be representatives for the left 
cosets of $HZ$ in $G$, and set 
	\[ v = \sum_{i=1}^m\,\sum_{t\in Z_0}\,\sum_{s=0}^{p^\ell-1} 
	s\lambda \cdot(tx^sg_iH) \in \5V. \]
Let $z_0\in Z_0$ and $0<r<p^\ell$ be such that $z=z_0x^r$. Then 
	\[ zv = \sum_{i=1}^m\,\sum_{t\in Z_0}\,\sum_{s=0}^{p^\ell-1} 
	s\lambda \cdot(tz_0x^{s+r}g_iH) = v - r\lambda\cdot \sigma_G, \]
and $[z,v]\ne0$ since $r\lambda\ne0$. 

For each $g\in G$, let $z_1,\dots,z_m\in Z_0$ and $r_1,\dots,r_m\in\Z$ 
be such that for each $i$, $gg_iH=z_jx^{r_j}g_jH$ for some $j$. Then 
	\[ gv = \sum_{i=1}^m\,\sum_{t\in Z_0}\,\sum_{s=0}^{p^\ell-1} 
	s\lambda \cdot(tx^sgg_iH) 
	= \sum_{j=1}^m\,\sum_{t\in Z_0}\,\sum_{s=0}^{p^\ell-1} 
	s\lambda \cdot(tz_jx^{s+r_j}g_jH) = v - \sum_{j=1}^m
	r_j\lambda\cdot g_j\sigma_Z, \]
and so $[g,v]\in C_{\5V}(Z)=V_2$. Thus $v\in V$, finishing the proof that 
$Z$ acts faithfully on $V$. 

Since $[Z,[G,V]]=1$ by definition and $[Z,G]=1$, we have $[G,[Z,V]]=1$ by 
the three-subgroup lemma (see \cite[Theorem 2.2.3]{Gorenstein}). Hence 
$[Z,V]\le V_0$, so $Z$ acts trivially on $V/V_0$. 

If $V_1<V_2$, then $G$ acts trivially on $V_2/V_1$ and on $V/V_2$, and 
hence acts trivially on $V/V_1$ (recall $G$ is generated by $p'$-elements). 
So $V/V_1=(V_2/V_1)\times(V'/V_1)$ for some $\Z/p^kG$-submodule $V'<V$ 
containing $V_1$ with $V'/V_1\cong V/V_2$. Also, $Z$ acts faithfully on 
$V'$ since it acts faithfully on $V=V'+V_2$ and trivially on $V_2$, so $G$ 
acts faithfully on $V'$ since $G/Z$ acts faithfully on $[G,V_2]\le 
V_1=V'\cap V_2$.
\end{proof}

For example, when $p=2$, $G=2M_{12}$, $Z=Z(G)\cong C_2$, and $H\cong 
M_{11}$, then by Example \ref{ex:pG-ext}, there is a 
$12$-dimensional faithful $\F_2G$-module $V$ with submodules $V_0<V_1<V$, 
where $\dim(V_0)=1$, $\dim(V_1)=11$, $Z$ acts trivially on $V_1$ and on 
$V/V_0$, and where $V_1$ has index two in the $12$-dimensional permutation 
module for $G/Z\cong M_{12}$. 

There are much more general ways to construct faithful $\Z/p^kG$-modules 
$V$ with $V_0<V_1<V$ as in Lemma \ref{l:pG-rep}, starting with a given 
$\Z/p^k\4G$-module $V_1$ ($\4G=G/Z$). But the ones we have found all seem 
to require certain conditions on $H^2(\4G;V_1)$ to hold.

We end the section with the following, more technical lemma needed in 
Section \ref{s:A<|F}.

\begin{Lem} \label{l:C(A/A0)G}
Let $A$ be a finite abelian group, and fix $\alpha\in\Aut(A)$. Let $A_0\le 
A$ be such that $\alpha(A_0)=A_0$. Then 
$|C_{A/A_0}(\alpha)|\le|C_A(\alpha)|$. 
\end{Lem}

\begin{proof} Set $G=\gen{\alpha}\le\Aut(A)$. The short exact sequence 
$0\to A_0\too A\too A/A_0\to0$ induces an exact sequence in cohomology
	\[ 0 \Right2{} C_{A_0}(G) \Right3{} C_A(G) \Right3{} C_{A/A_0}(G) 
	\Right3{} H^1(G;A_0) \Right3{} \dots, \]
and hence 
	\[ |C_A(G)| \ge |C_{A/A_0}(G)|\cdot 
	|C_{A_0}(G)| \big/ |H^1(G;A_0)|. \]
Since $G=\gen{\alpha}$ and $A_0$ is finite, we have 
$|H^1(G;A_0)|=|H^2(G;A_0)|$ where $H^2(G;A_0)$ is a quotient group of 
$C_{A_0}(G)$ (see \cite[Theorem 6.2.2]{Weibel}). So 
$|C_A(G)|\ge|C_{A/A_0}(G)|$. 
\end{proof}

%%%%%%%%%%%%%%%%%%%%%%%%%%%%

\section{The Golay modules for \texorpdfstring{$M_{22}$ and $M_{23}$}
{M22 and M23}}
\label{s:Todd-F2}

We now apply results in Section 2 to prove that the Golay modules (i.e., 
dual Todd modules) for $M_{22}$ and $M_{23}$ are not fusion realizable in 
the sense of Definition \ref{d:realize}. We do this by showing that 
$\RR[+]{T}{A}=\emptyset$ (see Definition \ref{d:K&R}) whenever 
$T\in\syl2{M_n}$ ($n=22$ or $23$) and $A$ is the Golay module of $M_n$. 

We first set up our notation for handling these groups and modules. The 
notation used here for doing that is based mostly on that used by Griess 
\cite[Chapter 4--5]{Griess}. 

For a finite set $I$ and a field $K$, let $K^I$ be the 
vector space of maps $I\too K$, with canonical basis
$\{e_i\,|\,i\in I\}$. Let 
	\[ \Perm_I(K) \le \Mon^*_I(K) \le \Aut^*(K^I) \]
be the groups of permutation automorphisms, semilinear monomial 
automorphisms, and all semilinear automorphisms, respectively (i.e., linear 
with respect to some field automorphism of $K$). Thus if $|I|=n$, then 
$\Perm_I(K)\cong\Sigma_n$ and $\Mon^*_I(K)\cong 
(K^\times)^n\rtimes(\Sigma_n\times\Aut(K))$. Let 
	\[ \pi = \pi_{I,K} \: \Mon^*_I(K) \Right4{} \Perm_I(K) \]
be the canonical projection that sends a monomial automorphism to the 
corresponding permutation automorphism; thus $\Ker(\pi_{I,K})$ is the group of 
semilinear automorphisms that send each $Ke_i$ to itself.

More concretely, set 
	\[ I=\{1,2,3,4,5,6\} \qquad\textup{and}\qquad \Omega=\F_4\times I. 
	\]
Thus $\F_2^\Omega$ and $\F_4^I$ are the vector spaces of functions 
$\Omega\too\F_2$ and $I\too\F_4$, respectively. We also identify $\F_4^I$ 
with the space of 6-tuples in $\F_4$. Fix $\omega\in\F_4\sminus\F_2$, and 
let $(x\mapsto\4x)$ be the field automorphism of $\F_4$ of order $2$. Thus 
$\F_4=\{0,1,\omega,\4\omega\}$, and $\4x=x^2$ for $x\in\F_4$. 

Let $\scrh\subseteq\F_4^I$ be the hexacode subgroup: 
	\beqq \scrh=\Gen{(\omega,\4\omega,\omega,\4\omega,\omega,\4\omega), 
	(\4\omega,\omega,\4\omega,\omega,\omega,\4\omega), 
	(\4\omega,\omega,\omega,\4\omega,\4\omega,\omega), 
	(\omega,\4\omega,\4\omega,\omega,\4\omega,\omega) }_{\F_4} . 
	\label{e:H-gens} \eeqq
Thus $\scrh$ is a 3-dimensional $\F_4$-linear subspace of $\F_4^I$. When 
making computations, we will frequently refer to the following elements in 
$\scrh$:
	\beqq h_1=(1,1,1,1,0,0), \qquad h_2=(1,1,0,0,1,1), \qquad 
	h_3 = (\omega,\4\omega,1,0,1,0). \label{e:h1-h3} \eeqq

\begin{Not} \label{not:Gamma-action}
Let the group $\Gamma\defeq\F_4^I\rtimes\Mon_I^*(\F_4)$ act on 
$\Omega=\F_4\times I$ in the 
usual way: $\F_4^I$ acts via translation, $(\F_4^\times)^I$ acts via 
multiplication in each coordinate, $\Perm_I(\F_4)$ permutes the 
coordinates, and $\phi\in\Aut(\F_4)$ sends $(c,i)$ to $(\4c,i)$. This in 
turn induces an action on $\F_2^\Omega$, where $g\in\Gamma$ sends an 
element $e_{(c,i)}$ to $e_{g(c,i)}$. Equivalently, for 
$\xi\in\F_2^\Omega$ and $(c,i)\in\Omega$, define $g(\xi)$ by 
$(g(\xi))(c,i)=\xi(g^{-1}(c,i))$. 

As special cases, $\trs\eta\in\Aut(\F_2^\Omega)$ will denote translation by 
$\eta\in\F_4^I$, and $\ttt(\alpha)\in\Aut(\F_2^\Omega)$ will be the 
automorphism induced by $\alpha\in\Mon_I^*(\F_4)$. Thus 
	\[ \trs\eta(\xi)(c,i)=\xi(c-\eta(i),i) \qquad\textup{and}\qquad 
	\ttt(\alpha)(\xi)(c,i)=\xi(\alpha^{-1}(c,i)). \]
\end{Not}

Now set
	\[ \Aut^*(\scrh) \defeq 
	\bigl\{\alpha\in\Mon^*_I(\F_4) \,\big|\, 
	\alpha(\scrh)=\scrh\bigr\}. \]
By \cite[Proposition 4.5.ii]{Griess}, $\Aut^*(\scrh)\cong3\Sigma_6$. In 
other words, each permutation of $I$ is the image of some automorphism of 
$\scrh$, unique up to multiplication by $u\cdot\Id$ for some $u\in\F_4^\times$. 
More explicitly, $\Aut^*(\scrh)$ is generated by the subgroup 
	\[ \Aut^*_0(\scrh) = \Gen{(1\,2)(3\,4),(1\,2)(5\,6),(1\,3\,5)(2\,4\,6), 
	(1\,3)(2\,4), (1\,2)(3\,4)(5\,6)\phi} \cong \Sigma_4\times C_2, \]
where $\phi$ is the field automorphism 
$\phi(x_1,\dots,x_6)=(\4{x_1},\dots,\4{x_6})$, together with the elements 
	\[ \omega\cdot\Id \qquad\textup{and}\qquad
	\alpha=(1\,2\,3)\cdot\diag(1,1,1,1,\4\omega,\omega). \]

We refer to \cite[Definition 5.15]{Griess} for a definition of the Golay 
code $\scrg\le\F_2^\Omega$. Here, rather than repeat that definition, we 
give a set of generators. Define $\Grf\:\F_4^I\too\F_2^\Omega$ by setting 
	\[ \Grf(\xi) = \sum\nolimits_{i\in I} e_{(\xi(i),i)} \]
(the ``graph'' of $\xi$). Define elements in $\F_2^\Omega$: 
	\[ C_i = \sum_{c\in\F_4} e_{(c,i)} \quad \textup{(for $i\in I$)} 
	\qquad\textup{and}\qquad
	\grf{h} = \Grf(h)+\Grf(0) \quad \textup{(for $h\in\F_4^I$),} \]
and also $C_{ij}=C_i+C_j$ for distinct $i,j\in I$ and 
$C_{1234}=C_{12}+C_{34}$. Then $C_i+\Grf(0)$ and $\grf{h}$ are in $\scrg$ 
for all $i\in I$ and all $h\in\scrh$. From the ``standard basis'' for 
$\scrg$ given in \cite[5.35]{Griess}, we see that 
	\[ \scrg = \Gen{C_i+\Grf(h) \,\big|\, i\in I,~ h\in\scrh} = 
	\Gen{C_i+\Grf(0),\grf{h}, \,\big|\, i\in I,~ h\in\scrh} . \]
This is a 12-dimensional subspace of $\F_2^\Omega$, with basis consisting 
of the six elements $C_i+\Grf(0)$ for $i\in I$, together with six elements 
$\grf{h}$ for $h$ in any given $\F_2$-basis of $\scrh$. 
By \cite[Theorem 5.8]{Griess}, the weight of each element in $\scrg$ is 
$0$, $8$, $12$, $16$, or $24$.

Define $\Mat24$ to be the group of permutations of 
$\Omega$ that preserve $\scrg$, and set $\dTodd24=\scrg/\gen{e_\Omega}$: 
its Golay module. Also, define 
	\[ \Delta_1=\{(0,6)\} \qquad\textup{and}\qquad 
	\Delta_2=\{(0,6),(1,6)\}, \]
and for $i=1,2$ set 
	\[ \Mat{24}{-i} = C_{\Mat24}(\Delta_i) \qquad\textup{and}\qquad
	\dTodd{24}{-i} = \bigl\{ \xi\in\scrg \,\big|\, 
	\supp(\xi)\cap\Delta_i=\emptyset \bigr\}. \]
Thus $\dim(\dTodd24)=\dim(\dTodd23)=11$, while $\dim(\dTodd22)=10$.

Define permutations 
$\ttt_{ij},\trs{h}\in\Sigma_\Omega$ for $i\ne j$ in $I$ and $h\in\F_4^I$ by 
letting $\ttt_{ij}$ exchange the $i$-th and $j$-th columns and letting 
$\trs{h}$ be translation by $h$. More precisely,
	\[ \ttt_{ij}(c,k) = (c,\sigma(k))~ \textup{where}~ 
	\sigma=(i\,j)\in\Sigma_6
	\qquad\textup{and}\qquad
	\trs{h}(c,i)=(c+h(i),i). \]
Then $\trs{h}\in\Mat24$ for all $h\in\scrh$. By the above description of 
$\Aut^*_0(\scrh)\le\Aut^*(\scrh)$, the elements $\ttt_{12}\ttt_{34}$, 
$\ttt_{12}\ttt_{56}$, and $\ttt_{13}\ttt_{24}$ all lie in $\Mat24$.

\begin{Not} \label{n:M22-23}
Fix $n=22$ or $23$. Set $\Gamma=\Mat{}n$, and define subgroups 
	\begin{align*} 
	T &= \Gen{\trsh1, \trsh[\omega]1, \trsh3, \trsh[\omega]3, 
	\ttt_{12}\ttt_{34}, \ttt_{13}\ttt_{24}, 
	\ttt_{12}\phi} \in \syl2\Gamma \\
	H_1 &= \gen{\trsh1,\trsh[\omega]1,\ttt_{12}\ttt_{34}, 
	\ttt_{13}\ttt_{24}} \\ 
	H_2 &= \gen{\trsh1,\trsh[\omega]1, \trsh3, \trsh[\omega]3 } .
	\end{align*}
\end{Not}

In the next lemma, we list the basic properties of these subgroups 
that will be needed.

\begin{Lem} \label{l:hexad.grp}
Assume Notation \ref{n:M22-23}, with $n=22$ or $23$. Then $H_1$ and $H_2$ 
are the only subgroups of $T$ isomorphic to $E_{16}$. If we set 
$A=\dTodd{}n$, then 
	\begin{align*} 
	[\trsh1,A]&=\gen{C_{12},C_{13},C_{14},\grfh1}\cong E_{16}, \\
	C_A(H_1)&=C_A(T)=\gen{C_{1234}}, \\ 
	C_A(H_2)&=\gen{C_{12},C_{13},C_{14},C_{15}}\cong E_{16}. 
	\end{align*}
\end{Lem}

\begin{proof} The first statement is well known and easily checked. Note, 
for example, that $T/H_1\cong D_8$, and that $C_{H_1}(x)$ has rank $2$ for 
$x\in T\sminus H_1$. So if $E_{16}\cong H\le T$ and $H\ne H_1$, then 
$HH_1=H_1\gen{\trsh3,\trsh[\omega]3}$ or 
$H_1\gen{\trsh3,\ttt_{12}\phi}$, and from this one easily reduces to 
the case $H=H_2$. (Note that all elements of order $2$ in $H_1H_2$ lie in 
$H_1\cup H_2$.) 

The statements about commutators and centralizers follow from Tables 
\ref{tbl:[x,a]} and \ref{tbl:[ch1,x]}. 
\end{proof}

\begin{Table}[ht]
\[ \renewcommand{\arraystretch}{1.2} \addtolength{\arraycolsep}{1mm}
\begin{array}{r|ccccccc} 
x & C_{1234} & C_{12} & C_{13} & C_{15} & \grfh1 & \grfh2+C_{56} & 
\grf{h_3+\omega h_2}+C_{56} \\\hline
[\trsh1,x] & 0 & 0 & 0 & 0 & 0 & 0 & 0 \\ 
{}[\trsh[\omega]1,x] & 0 & 0 & 0 & 0 & C_{1234} & C_{12} & C_{23} \\ 
{}[\trsh3,x] & 0 & 0 & 0 & 0 & C_{12} & C_{12} & C_{25} \\
{}[\trsh[\omega]3,x] & 0 & 0 & 0 & 0 & C_{13} & C_{15} & C_{35} \\
{}[\ttt_{12}\ttt_{34},x] & 0 & 0 & C_{1234} & C_{12} & 0 & 0 & \grfh1 \\
{}[\ttt_{13}\ttt_{24},x] & 0 & C_{1234} & 0 & C_{13} & 0 & \grfh1 & \grfh1 \\
{}[\ttt_{12}\phi,x] & 0 & 0 & C_{12} & C_{12} & 0 & 0 & \grfh2+C_{56} 
\end{array} \]
\caption{Table of commutators $[g,x]=g(x)-x$ for $g\in T$ and 
$x\in\dTodd23$. The first six elements in the top row form a 
basis for $C_{\dTodd22}(\trsh1)$, and together with the seventh they form a 
basis for $C_{\dTodd23}(\trsh1)$.} 
\label{tbl:[x,a]} 
\end{Table}

\begin{Table}[ht]
\[ \renewcommand{\arraystretch}{1.5}
\begin{array}{r|cccc} 
x & \grfh[\omega]1 & \grfh3 & \grfh[\omega]3 & \Grf(\omega h_2)+C_1 
\\\hline
[\trsh1,x] & C_{1234} & C_{12} & C_{13} & \grfh1+C_{12} 
\end{array}
\]
\caption{The classes of these four elements $x$ form a basis for 
$\dTodd{}n/C_{\dTodd{}n}(\trsh1)$.} 
\label{tbl:[ch1,x]}
\end{Table}

%%%%%%%%%%%%%%%%%%%%%%%%%%%%%%%%%%%%%%%%%%

\section{The 6-dimensional module for \texorpdfstring{$3M_{22}$}{3M22}}
\label{s:3M22}

\newcommand{\hh}{\mathfrak{h}}
\renewcommand{\3}{\textbf{\underline3}}

We again fix an element $\omega\in\F_4\sminus\F_2$, and let $(a\mapsto\4a)$ 
denote the field automorphism of $\F_4$. Thus $\F_4=\{0,1,\omega,\4\omega\}$. 
We also use the bar over matrices to denote the field automorphism applied 
to the entries; i.e., $\4{(a_{ij})}=(\4{a_{ij}})$. Let $\Tr\:\F_4\too\F_2$ be 
the trace: $\Tr(a)=a+\4a$.

Set $V=\F_4^3$ and $A=\F_4^6$, where elements of $V$ are written as column 
matrices $\colthree{a}bc$ for $a,b,c\in\F_4$, and elements of $A$ are 
written as column matrices $\coltwo{u}v$ for $u,v\in V$. Let 
$\gen{-,-}$ be the hermitian form on $A$ defined by 
	\[ \left\langle \Coltwo{u}v, \Coltwo{x}y \right\rangle 
	= \Tr(u^t\4y+v^t\4x). \]

The description here of the action of $\Gamma=3M_{22}$ on $A$ is based on 
that in \cite[Chapter 2]{bensonthesis} and in \cite[p. 39]{atlas}, 
originally due to Benson and others. An element denoted 
\fbox{$\begin{smallmatrix}r&s&t\\x&y&z\end{smallmatrix}$} in 
\cite{bensonthesis} or \texttt{(rx~sy~tz)} in \cite{atlas} is written here 
$\Coltwo{u}v$ where $u=\colthree{r}st$ and $v=\colthree{r+x}{s+y}{t+z}$.

For $i,k=1,2,3$ and $j=1,2$, define 
	\[ b_{ijk} = \begin{cases} 
	\omega^j & \textup{if $i=k$} \\
	1 & \textup{if $i\ne k$,}
	\end{cases} \qquad\textup{and}\qquad
	b_{ij} = \Colthree{b_{ij1}}{b_{ij2}}{b_{ij3}}\in V; \]
and set $\scrb=\{\gen{b_{ij}}\,|\,i=1,2,3,~j=1,2\}$. The following lemma 
is easily checked.

\begin{Lem} \label{l:B&U}
Consider the hermitian form $\hh\:V\times V\too\F_4$ defined by 
$\hh(v,w)=\4v^tw$. Define elements $u_1,\dots,u_6\in V$ by setting 
	\[ u_1=\colthree100, \quad u_2=\colthree010,\quad 
	u_3=\colthree001,\quad u_4=\colthree111,\quad 
	u_5=\colthree1\omega{\4\omega},\quad 
	u_6=\colthree1{\4\omega}\omega, 
	\]
and set $\scru=\{\gen{u_i}\,|\,1\le i\le6\}$. Then the members of $\scru$ 
are the only $1$-dimensional subspaces of $V$ not orthogonal to any member 
of $\scrb$, and the members of $\scrb$ are the only $1$-dimensional 
subspaces of $V$ not orthogonal to any member of $\scru$. Hence for 
$D\in\GL_3(4)$, the action of $D$ on $V$ permutes the members of $\scru$ if 
and only if the action of $\4D^t$ on $V$ permutes the members of $\scrb$.
\end{Lem}

Define matrices 
	\[ M_{10} = \mxthree000010001, \quad 
	M_{20} = \mxthree100000001, \quad 
	M_{01} = \mxthree011101110, \quad 
	M_{02} = \mxthree0\omega{\4\omega}{\4\omega}0\omega\omega{\4\omega}0,\]
and set $M_{00}=0$, $M_{03}=M_{01}+M_{02}$, $M_{30}=M_{10}+M_{20}$, and 
$M_{ij}=M_{i0}+M_{0j}$ for $i,j=1,2,3$. In other words, if we set 
$\3=\{0,1,2,3\}$ and regard it as an elementary abelian $2$-group via 
bitwise sum, then $((i,j)\mapsto M_{ij})$ is a homomorphism from 
$\3\times\3$ to $M_3(\F_4)$. 

Finally, set 
	\[ N_{ij}=I+M_{ij} \qquad ((i,j)\in\3\times\3). \]
Note that 
	\beqq N_{i0}=\4{u_i}u_i^t \qquad\textup{and}\qquad 
	N_{0i}=\4{u_{i+3}}u_{i+3}^t \qquad\textup{for all $i=1,2,3$.} 
	\label{e:Ai0} \eeqq

\begin{Not} \label{n:3M22}
Define maximal isotropic subspaces $X_{ij}\le A$ (for $i,j=0,1,2,3$) and 
$Y_{ij}\le A$ (for $i=1,2,3$ and $j=1,2$) as follows: 
	\[ X_{ij} = \left\{ \left. \Coltwo{N_{ij}v}v
	\,\right|\, 
	v\in V \right\} \quad\textup{and}\quad 
	Y_{ij} = \left\{\left. \Coltwo{u}{b_{ij}\4{b_{ij}}^tu} \,\right|\, u\in 
	V \right\} . \]
Set $\scrx=\{X_{ij}\,|\,i,j=0,1,2,3\}$ and 
$\scry=\{Y_{ij}\,|\,i=1,2,3,~j=1,2\}$. Let $\Gamma\le\Aut(A)$ be the group 
of unitary automorphisms of $A$ that permute the members of 
$\scrx\cup\scry$. 
\end{Not}

The members of $\scrx\cup\scry$ are all totally isotropic since the 
matrices $N_{ij}$ and $b_{ij}\4{b_{ij}}^t$ are hermitian for all $i,j$. 
Following \cite{bensonthesis} and \cite{atlas}, we arrange them 
diagrammatically as follows:
	\beqq \renewcommand{\arraystretch}{1.5}
	\renewcommand{\arraycolsep}{2mm}
	\begin{array}{|cc|cc|cc|} \hline
	 & & X_{00} & X_{01} & X_{02} & X_{03} \\
	Y_{12} & Y_{11} & X_{10} & X_{11} & X_{12} & X_{13} \\\hline
	Y_{22} & Y_{21} & X_{20} & X_{21} & X_{22} & X_{23} \\
	Y_{32} & Y_{31} & X_{30} & X_{31} & X_{32} & X_{33} \\\hline
	\end{array} \label{e:MOG} \eeqq

\begin{Not} \label{n:3M22b}
For $M\in M_3(\F_4)$ and $D\in\GL_3(\F_4)$, 
define $\varphi_M,\psi_D\in\Aut(A)$ by setting
	\[ \varphi_M\left(\Coltwo{u}v\right) = \Mxtwo{I}C0I\Coltwo{u}v 
	\qquad\textup{and}\qquad 
	\psi_D\left(\Coltwo{u}v\right) = \Mxtwo{D}00{\4D^{-t}}\Coltwo{u}v 
	\]
where $(-)^{-t}$ means transpose inverse. Set 
	\[ D_0 = \mxthree100001010, \qquad D_1 = \mxthree100101110,\qquad
	D_2 = \mxthree011101001,\qquad D_3 = 
	\mxthree1000{\omega}000{\4\omega}; \]
and set 
	\beq \mu_{ij}=\varphi_{M_{ij}} \qquad\textup{and}\qquad 
	\delta_i=\psi_{D_i} \tag{for $i,j=0,1,2,3$}. \eeq
Also, define the following subgroups of $\Aut_{\F_4}(A)$ (in fact, of 
$\Gamma$): 
	\begin{align*} 
	H &= N_\Gamma(\scry) = N_\Gamma(\scrx) 
	& P_1 &= \{\mu_{ij}\,|\,i,j=0,1,2,3\} \\
	H_0 &= C_\Gamma(\scry) 
	& P_2 &= \gen{\mu_{10},\mu_{01},\delta_0,\delta_1} \\
	\Gamma_0 &= C_\Gamma(\scrx\cup\scry) 
	& T &= P_1P_2\gen{\delta_2} = P_1\gen{\delta_0,\delta_1,\delta_2} . 
	\end{align*}
\end{Not}

Note that $\varphi_M$ is unitary whenever $\4M^t=M$, and $\psi_D$ is 
unitary for all $D\in\GL_3(4)$. In particular, the $\mu_{ij}$ and the 
$\delta_i$ are all unitary. 

Most of the information about $\Gamma$ and its action on $A$ in the following 
lemma is well known and implicit in Chapter 2 of \cite{bensonthesis}, but 
we try here to make more explicit some of the details in the proofs.

\begin{Lem} \label{l:3M22c}
Set $A_0=\left\{\left.\coltwo{w}0\,\right|\,w\in V\right\}$. 
Set $\Delta=\gen{D_0,D_1,D_2,D_3}\le\GL_3(4)$, and set 
$\psi_\Delta=\gen{\delta_0,\delta_1,\delta_2,\delta_3} 
=\{\psi_D\,|\,D\in\Delta\}\le\Aut(A)$. Then 
\begin{enuma} 

\item $\Gamma\cong3M_{22}$ and $T\in\syl2\Gamma$; 

\item $\Delta\cong\psi_\Delta\cong3A_6$; 

\item $H_0=P_1\times\Gamma_0$ where 
$P_1=\bigl\{\varphi\in\Gamma\,\big|\,\varphi|_{A_0}=\Id\bigr\} \cong E_{16}$ and 
$\Gamma_0=\gen{\omega\cdot\Id_A}$; and 

\item $H=\{\varphi\in\Gamma\,|\,\varphi(A_0)=A_0\} = P_1\psi_\Delta$.

\end{enuma}
\end{Lem}

\begin{proof} For each $i=1,2,3$ and $j=1,2$, 
	\beqq Y_{ij}\cap A_0 = \{\coltwo{u}0 \,|\, u\in V,~ \4{b_{ij}}^tu=0 
	\} = \{\coltwo{u}0 \,|\, u\in b_{ij}^\perp \} \label{e:Yij.A0} 
	\eeqq
in the notation of Lemma \ref{l:B&U}. Thus 
$\dim_{\F_4}(Y\cap A_0)=2$ for $Y\in\scry$, and distinct members of 
$\scry$ have distinct intersections with $A_0$. So for each pair $Y\ne Y'$ in 
$\scry$, we have $Y\cap Y'\le A_0$ where $\dim(Y\cap Y')=1$, and the set of 
all such intersections generates $A_0$. 

Thus each $\varphi\in H$ sends $A_0$ to itself. If $\varphi\in H_0$, then 
$\varphi$ sends each of the 1-dimensional subspaces $Y\cap Y'$ to itself 
(for $Y\ne Y'$ in $\scry$), and hence 
$\varphi|_{A_0}\in\gen{\omega\cdot\Id_{A_0}}$. 

By definition, $X\cap A_0=0$ for each $X\in\scrx$. So if $\varphi\in\Gamma$ 
is such that $\varphi(A_0)=A_0$, then $\varphi$ permutes the members of 
$\scrx$ and those of $\scry$, and hence lies in $H$. If $\varphi|_{A_0}\in 
\gen{\omega\cdot\Id_{A_0}}$, then since the intersections $Y\cap A_0$ for 
$Y\in\scry$ are all distinct, $\varphi$ sends each member of $\scry$ to 
itself and hence lies in $H_0$. To summarize, we have now shown that 
	\beqq H = \bigl\{\varphi\in\Gamma\,\big|\,\varphi(A_0)=A_0 \bigr\} 
	\qquad\textup{and}\qquad
	H_0 = \bigl\{\varphi\in\Gamma\,\big|\,
	\varphi|_{A_0}\in\gen{\omega\cdot\Id_{A_0}} \bigr\}. \label{e:H,H0} 
	\eeqq

\smallskip

\noindent\textbf{(b) } Each of the matrices $D_i$ for $i=0,1,2,3$ permutes the 
members of $\scru=\{u_i\,|\,1\le i\le6\}$, and does so via the permutations 
	\beqq D_0:~(2\,3)(5\,6), \qquad D_1:~(1\,4)(2\,3), \qquad 
	D_2:~(1\,2)(3\,4), \qquad D_3:~(4\,5\,6). \label{e:Di-perm} \eeqq
These generate the group of all even permutations of the set $\scru$. In 
particular, there is a matrix $D_4\in\Delta$ that induces the permutation 
$(1\,2\,3)$, and by considering its action on the $u_i$ for $1\le i\le4$, 
we see that $D_4=\mxthree00rr000r0$ for some $r\in\F_4^\times$. 

We claim that 
	\beqq \Delta = \{D\in\GL_3(4) \,|\, D(\scru)=\scru \}. 
	\label{e:Delta(U)} \eeqq
To see this, assume $D\in\GL_3(4)$ permutes the $\gen{u_i}$. Since all even 
permutations of $\scru$ are realized by elements in $\Delta$, there is 
$D'\equiv D$ (mod $\Delta$) that 
sends each of the subspaces $\gen{u_1}$, $\gen{u_2}$, 
$\gen{u_3}$, $\gen{u_4}$ to itself. But then $D'$ must have the form 
$s\cdot I$ for $s\in\F_4^\times$. Since 
	\[ \mxthree{\omega}000\omega000\omega = 
	\left[ \mxthree1000{\omega}000{\4\omega},\mxthree00rr000r0\right] 
	\in [\Delta,\Delta], \]
this proves \eqref{e:Delta(U)}, and also shows that $\Delta\cong3A_6$. 

The isomorphism $\psi_\Delta\cong\Delta$ follows directly from the 
definitions.

\smallskip

\noindent\textbf{(c) } We first check, for each $i,j=0,1,2,3$, $k=1,2,3$, 
and $\ell=1,2$, that $\mu_{ij}(Y_{k\ell})=Y_{k\ell}$. This means 
showing, for $u\in V$, that 
	\[ b_{k\ell}\4{b_{k\ell}}^tu = 
	b_{k\ell}\4{b_{k\ell}}^t(u+M_{ij}b_{k\ell}\4{b_{k\ell}}^tu); \]
i.e., that $\4{b_{k\ell}}^tM_{ij}b_{k\ell}=0$. It suffices to do this when 
$ij=0$ and $(i,j)\ne(0,0)$. In all such cases, by \eqref{e:Ai0}, there is 
$\gen{c_{ij}}\in\scru$ such that $c_{ij}\4{c_{ij}}^t=I+M_{ij}$. So it 
suffices to show that 
	\[ (\4{b_{k\ell}}^tc_{ij})\cdot\4{(\4{b_{k\ell}}^tc_{ij})} = 
	\4{b_{k\ell}}^tb_{k\ell} = 1; \]
equivalently, that $b_{k\ell}\not\perp c_{ij}$ --- which follows from Lemma 
\ref{l:B&U}. 

For the same automorphism $\mu_{ij}$ with matrix $\mxtwo{I}{M_{ij}}0I$, 
an element $\coltwo{N_{k\ell}u}u\in X_{k\ell}$ is sent to 
$\coltwo{N_{k\ell}u+M_{ij}u}u$. Since $N_{k\ell}+M_{ij}=N_{k+i,\ell+j}$ 
where sums of indices are taken bitwise, this shows that 
$\mu_{ij}(X_{k\ell})=X_{k+i,\ell+j}$. So $\mu_{ij}$ permutes the 
members of $\scrx$, finishing the proof that $\mu_{ij}\in H_0\le\Gamma$.

Conversely, for each $\varphi\in\Gamma$ such that $\varphi|_{A_0}=\Id$, 
$\varphi$ induces the identity on $A/A_0$ since it is unitary and $A_0$ is 
a maximal isotropic subgroup, so $\varphi$ has matrix $\mxtwo{I}M0I$ for 
some $M\in M_3(\F_4)$. Thus $\varphi=\varphi_M$ (see Notation 
\ref{n:3M22b}). Let $(i,j)$ be such that $\varphi(X_{00})=X_{ij}$; then 
$N_{00}+M=I+M=N_{ij}$, so $M=M_{ij}$, and $\varphi=\mu_{ij}\in P_1$. We 
now conclude that 
	\[ P_1 = \bigl\{\varphi\in\Gamma\,\big|\, \varphi|_{A_0}=\Id 
	\bigr\}. \]

By \eqref{e:H,H0}, $\varphi\in H_0$ implies that 
$\varphi|_{A_0}\in\gen{\omega\cdot\Id_{A_0}}$, and hence that $\varphi\in 
P_1\times\gen{\omega\cdot\Id_A}$. Thus $H_0\le 
P_1\times\gen{\omega\cdot\Id_A}$, and we already proved the opposite 
inclusion. Also, $\gen{\omega\cdot\Id_A}\le\Gamma_0\le H_0$, and 
$\Gamma_0\cap P_1=1$ since $P_1$ acts faithfully on $\scrx$. So 
$\Gamma_0=\gen{\omega\cdot\Id_A}$.

\smallskip

\noindent\textbf{(d) } Fix $D\in\Delta$; we will show that $\psi_D\in H$. 
Let $\rho_D\:M_3(\F_4)\too M_3(\F_4)$ be the homomorphism 
$\rho_D(M)=DM\4D^t$. Since $D$ permutes the members of $\scru$ by 
\eqref{e:Di-perm}, $\rho_D$ permutes the set
	\[ \{u\4u^t\,|\,\gen{u}\in\scru\} = 
	\{N_{10},N_{20},N_{30},N_{01},N_{02},N_{03}\} \]
(see \eqref{e:Ai0}). This, together with the relations 
$N_{ij}+N_{k\ell}+N_{mn}=N_{i+k+m,j+\ell+n}$ (where indices are added 
bitwise) shows that $\rho_D$ permutes the set of all $N_{ij}$ for 
$i,j=0,1,2,3$. (Note, for example, that $N_{00}=N_{10}+N_{20}+N_{30}$.) 
If $i,j,k,\ell$ are such that $\rho_D(N_{ij})=N_{k\ell}$, then 
	\[ \psi_D(X_{ij}) 
	= \left\{ \left. \Coltwo{DN_{ij}u}{\4D^{-t}u} 
	\,\right|\, u\in V \right\} 
	= \left\{ \left. \Coltwo{DN_{ij}\4D^tv}v \,\right|\, v\in V 
	\right\} = X_{k\ell}, \]
and thus $\psi_D$ permutes the members of $\scrx$.

By Lemma \ref{l:B&U} and since $D$ permutes the members of $\scru$, the 
matrix $\4D^t$ permutes the members of $\scrb$. So for each $i,j$ there is 
$k,\ell$ such that $\4D^tb_{k\ell}\in\gen{b_{ij}}$, and hence 
	\begin{multline*} 
	\psi_D(Y_{ij}) 
	= \left\{ \left. \Coltwo{Du}{\4D^{-t}b_{ij}\4{b_{ij}}^tu} 
	\,\right|\, u\in V \right\} 
	= \left\{ \left. \Coltwo{v}{\4D^{-t}b_{ij}\4{b_{ij}}^tD^{-1}v} 
	\,\right|\, v\in V \right\} \\[1mm]
	= \left\{ \left. \Coltwo{v}{b_{k\ell}\4{b_{k\ell}}^tv} 
	\,\right|\, v\in V \right\} = Y_{k\ell}.
	\end{multline*}
Thus $\psi_D$ also permutes the members of $\scry$, and it follows that 
$\psi_D\in H$. 

Conversely, for each $\eta\in H$, $\eta(A_0)=A_0$ by \eqref{e:H,H0}, and 
$\eta|_{A_0}$ permutes the subspaces $Y\cap A_0$ for $Y\in\scry$. So $\eta$ 
has matrix of the form $\mxtwo{D}X0{\4D^t}$, where $D$ permutes the 
subspaces $b^\perp\le V$ for all $\gen{b}\in\scrb$ by \eqref{e:Yij.A0}, and 
hence permutes the members of $\scru$ by Lemma \ref{l:B&U}. So by 
\eqref{e:Delta(U)}, there is $\delta\in\psi_\Delta$ such that 
$\eta|_{A_0}=\delta|_{A_0}$. Then $\delta^{-1}\eta\in P_1$ by (c), and 
$\eta\in P_1\psi_\Delta$. This finishes the proof that $H=P_1\psi_\Delta$. 

\smallskip

\noindent\textbf{(a) } Set $\Gamma^*=\Gamma/\Gamma_0$, regarded as a group 
of permutations of the set $\scrx\cup\scry$. By \cite[Theorem 
2.3]{bensonthesis}, $\Gamma^*$ acts 3-transitively on the set 
$\scrx\cup\scry$. It is well known (see, e.g., \cite[p. 235]{Pogorelov}) 
that the only finite groups that act $2$-transitively on a set of order 
$22$ are $M_{22}$, $A_{22}$, and their automorphism groups. So once we have 
shown that $T\in\syl2\Gamma$ and $|T|=2^7$, it will then follow that 
$\Gamma^*\cong M_{22}$, and that $\Gamma$ is a central extension of 
$\Gamma_0\cong C_3$ by $\Gamma^*$. 

Recall that $T=P_1\gen{\delta_0,\delta_1,\delta_2}$, where by 
\eqref{e:Di-perm}, the action of the $\delta_i$ on $\scry$ generates a 
subgroup of $\Sigma_6$ isomorphic to $D_8$. Hence $T/P_1\cong D_8$, and 
$|T|=2^7$. Alternatively, one can describe $T$ by looking at the subgroup 
of $\Aut(A_0)$ generated by restrictions of its elements.

Under the action of $\Gamma^*$, the stabilizer of a subspace 
$X\in\scrx\cup\scry$ acts $\F_4$-linearly on $X$. If $\varphi\in\Gamma$ is 
such that $\varphi|_X=\Id_X$, then $\varphi$ sends each member of 
$\scrx\cup\scry$ to itself since their intersections with $X$ are distinct, 
and hence $\varphi\in\Gamma_0$. The point stabilizer for the action of 
$\Gamma^*$ on $\scrx\cup\scry$ is thus isomorphic to a subgroup of $\PGL_3(4)$, 
and hence the order of $\Gamma^*$ divides 
$22\cdot|\PGL_3(4)|=2^7\cdot3^3\cdot5\cdot7\cdot11=3\cdot|M_{22}|$. So 
$T\in\syl2\Gamma$ and $\Gamma^*\cong M_{22}$. Finally, $\Gamma$ is the 
nonsplit central extension of $\Gamma_0\cong C_3$ by $\Gamma^*$ since it 
contains $\psi_D\cong3A_6$ by (b,d). 
\end{proof}

Thus $P_1=O_2(H)$, where $H\cong E_{16}\rtimes3A_6$ is a hexad subgroup of 
$\Gamma\cong3M_{22}$. One can also show that $P_2=O_2(K)$ where 
$K=N_\Gamma(\{Y_{11},Y_{12}\})$ is a duad subgroup of $\Gamma$. 
Equivalently, $K\cong C_3\times(E_{16}\rtimes\Sigma_5)$ is the group of 
elements of $\Gamma$ that permute the five $2\times2$ blocks in diagram 
\eqref{e:MOG}; i.e., send the four members of each such block to those in 
another block.

The next lemma collects some technical properties of the action of 
$\Gamma$ on $A$. 

\begin{Lem} \label{l:3M22b}
Let $\{e_1,e_2,\dots,e_6\}$ be the canonical basis for $A=\F_4^6$. Then for 
$P_1,P_2,T\le\Gamma$ and $\mu_{10}\in P_1\cap P_2$ as defined in 
Notation \ref{n:3M22b}, 
\begin{enuma} 
\item $P_1$ and $P_2$ are the only subgroups of $T$ isomorphic to $E_{16}$; 
and 
\item $C_A(\mu_{10})=\gen{e_1,e_2,e_3,e_4}$, 
$[\mu_{10},A]=\gen{e_2,e_3}$, $C_A(P_1)=\gen{e_1,e_2,e_3}$, 
$C_A(P_2)=\gen{e_2+e_3}$.
\end{enuma}
\end{Lem}

\begin{proof} For point (a), see Lemma \ref{l:hexad.grp}. Point (b) 
follows easily from the above descriptions of the actions. 
\end{proof}

%%%%%%%%%%%%%%%%%%%%%%%%%%%%%%%%%%%%%%%%%%%%


\begin{thebibliography}{BLO2}
\bibitem[K\i]{Kizmaz} Y. K{\i}zmaz, A generalization of Alperin's fusion 
theorem and its applications, arxiv 2204.11303
\end{thebibliography}
\end{document}